\documentclass[11pt]{article}
\usepackage{eurosym}
\usepackage{amsmath,amsfonts,amssymb,amsthm}
\usepackage{mathrsfs}
\usepackage{latexsym}
\usepackage[english]{babel}
\usepackage[usenames,dvipsnames]{xcolor}
\usepackage{comment}
\usepackage{appendix}
\usepackage{graphicx}

\renewenvironment{proof}[1][\proofname]{{\bfseries #1.} }{\qed}

\newtheorem{theorem}{Theorem}

\newtheorem{condition}[theorem]{Assumption}

\newtheorem{definition}[theorem]{Definition}

\newtheorem{lemma}[theorem]{Lemma}

\newtheorem{proposition}[theorem]{Proposition}
\newtheorem{remark}[theorem]{Remark}

\begin{document}

\title{\textsc{Fluctuations of Level Curves for Time-Dependent Spherical
Random Fields} }
\author{Domenico Marinucci \\
\emph{Dipartimento di Matematica, Universit\`{a} di Roma ``Tor Vergata''} \and Maurizia
Rossi \\
\emph{Dipartimento di Matematica e Applicazioni, Universit\`{a} di Milano-Bicocca}
\and Anna Vidotto \\
\emph{Dipartimento di Matematica e Applicazioni, Universit\`{a} di Napoli
``Federico II''}}
\date{{\footnotesize \today}}
\maketitle

\begin{abstract}
The investigation of the behaviour for geometric functionals of random
fields on manifolds has drawn recently considerable attention. In this
paper, we extend this framework by considering fluctuations over time for
the level curves of general isotropic Gaussian spherical random fields. We
focus on both long and short memory assumptions; in the former case, we show
that the fluctuations of $u$-level curves are dominated by a single
component, corresponding to a second-order chaos evaluated on a subset of
the multipole components for the random field. We prove the existence of
cancellation points where the variance is asymptotically of smaller order;
these points do not include the nodal case $u = 0$, in marked contrast with
recent results on the high-frequency behaviour of nodal lines for random
eigenfunctions with no temporal dependence. In the short memory case, we
show that all chaoses contribute in the limit, no cancellation occurs and a
Central Limit Theorem can be established by Fourth-Moment Theorems and a
Breuer-Major argument.

\smallskip


\smallskip

\end{abstract}

\allowdisplaybreaks


\section{Background and notation}

The analysis of level curves for random fields is a very classical topic in
stochastic geometry. In particular, many efforts have focussed
on the investigation of level-zero curves (i.e., nodal lines) in the case of
random eigenfunctions, in the high-frequency regime where eigenvalues are
assumed to diverge to infinity; see for instance \cite{Wig, MPRW16, NPR19}, or more generally \cite{Wig22} for a recent overview. In the
same high-energy regime, other functionals for random eigenfunctions
(including excursion area, the Euler-Poincar\'e characteristic, the number of
critical points) have also been widely investigated, see for instance \cite%
{M2022}; on the other hand, these same functionals have also been considered
by different authors in the asymptotic regime where the spatial domain of
the field is assumed as growing, notable examples being \cite{KL01} (for
level curves) and \cite{EL2016} (for the Euler-Poincar\'{e} characteristic).

Our purpose in this paper is to study the behaviour of level curves under a
different asymptotic regime than so far considered, namely for sphere
cross-time random fields and taking into account the averaged fluctuations
over time around the expected value (a similar framework was considered for
the case of the excursion area in \cite{MRV:20}). Our asymptotic results
share some analogies with the different settings that we mentioned above,
but they also show very important differences that we shall discuss below in
greater detail. While this paper only focusses on theoretical aspects, it is
really not difficult to envisage application areas where sphere cross-time
random fields emerge very naturally, some examples being atmospheric and
climate data (the sphere representing the surface of the Earth, see \cite%
{Ch:17}).

\subsection{Time-dependent spherical random fields}

We start by recalling the notion of space-time spherical random field, along
with the corresponding spectral representation, which allows the
characterization of long and short range dependence properties. Our
assumptions and discussion is close to the one that can be found in \cite%
{MRV:20}.

More precisely, let us take a probability space $(\Omega ,\mathfrak{F},%
\mathbb{P})$ and denote by $\mathbb{E}$ the expectation under $\mathbb{P}$:
all random objects in this manuscript are defined on this common probability
space, unless otherwise specified. Let $\mathbb{S}^{2}$ denote the
two-dimensional unit sphere with the round metric, usually written in the
form
\begin{equation*}
ds^{2}=d\theta ^{2}+\sin ^{2}\theta d\varphi ^{2}
\end{equation*}%
for standard spherical coordinates $(\theta ,\varphi )$, where $\theta $ is
the colatitude and $\varphi $ the longitude. A space-time real-valued
spherical random field $Z$ is a collection of real random variables indexed
by $\mathbb{S}^{2}\times \mathbb{R}$
\begin{equation}
{Z=\{Z(x,t),\,x\in \mathbb{S}^{2},\,t\in \mathbb{R}\}}  \label{eq1}
\end{equation}%
such that the function $Z:\Omega \times \mathbb{S}^{2}\times \mathbb{R}%
\rightarrow \mathbb{R}$ is $\mathfrak{F}\otimes \mathfrak{B}(\mathbb{S}%
^{2}\times \mathbb{R})$-measurable, {$\mathfrak{B}(\mathbb{S}^{2}\times
\mathbb{R})$ being the Borel $\sigma $-field of $\mathbb{S}^{2}\times
\mathbb{R}$}. The following condition is standard.

\begin{condition}
\label{basic} The space-time real-valued spherical random field $Z$ in %
\eqref{eq1} is

\begin{itemize}
\item Gaussian, i.e.~its finite dimensional distributions are Gaussian;

\item centered, that is, $\mathbb{E}[Z(x,t)] = 0$ for every $x\in \mathbb{S}%
^2$, $t\in \mathbb{R}$;

\item isotropic (in space) and stationary (in time), namely
\begin{equation}  \label{Gamma}
\mathbb{E}[Z(x,t) Z(y,s)] = \Gamma( \langle x,y\rangle, t-s)
\end{equation}
for every $x,y\in \mathbb{S}^2$, $t,s \in \mathbb{R}$, where $\langle x,
y\rangle$ denotes the standard inner product in $\mathbb{R}^3$ and $%
\Gamma:[-1,1]\times \mathbb{R }\to \mathbb{R}$ is a positive semidefinite
function;

\item mean square continuous, i.e. $\Gamma$ is continuous.
\end{itemize}
\end{condition}

Assumption \ref{basic} collects the common background with basically all the
previous literature, starting from \cite{adlertaylor}, on the geometry of
excursion sets of (time-varying) random fields on manifolds, see also \cite%
{BP17, LO13, MRV:20}. \emph{From now on we assume that $Z$ in \eqref{eq1}
satisfies Assumption \ref{basic}.} In order to ensure that our functional of
interest is well defined, we will need the following further condition.

\begin{condition}
\label{regular} The sample paths of the space-time real-valued spherical
random field $Z$ in \eqref{eq1} are a.s. $\mathcal{C}^1 (\mathbb{S}^2\times
\mathbb{R})$ functions.
\end{condition}

It is worth noticing that if the covariance kernel in (\ref{Gamma}) is a
twice continuously differentiable function, then there exists a modification
of $Z$ satisfying Assumption \ref{regular}.

\emph{From now on we assume that $Z$ in \eqref{eq1} satisfies also
Assumption \ref{regular}}; however, it is plausible that slightly weaker
regularity properties suffice. Note that the spatial gradient
\begin{equation*}
\nabla Z :=\lbrace \nabla Z(x,t), (x,t)\in \mathbb{S}^2\times \mathbb{R}%
\rbrace
\end{equation*}
is a centered Gaussian random field indexed by $\mathbb{S}^2\times \mathbb{R}
$ whose covariance kernel is the spatial Hessian of $\Gamma$ in (\ref{Gamma}%
).

Under Assumption \ref{basic} it is well known (see e.g. \cite[Theorem 3.3]%
{BP17} and \cite[Theorem 3]{MM18}) that the covariance function $\Gamma$ in %
\eqref{Gamma} of $Z$ in (\ref{eq1}) can be written as a uniformly convergent
series of the form
\begin{equation}  \label{eq2}
\Gamma(\eta, \tau) = \sum_{\ell=0}^{+\infty} \frac{2\ell+1}{4\pi}
C_\ell(\tau) P_\ell(\eta)\,,\quad (\eta, \tau) \in [-1,1] \times \mathbb{R}%
\,,
\end{equation}
where $\lbrace C_\ell, \ell\ge 0\rbrace$ is a sequence of continuous
positive semidefinite functions on the real line and $\lbrace P_\ell,
\ell\ge 0\rbrace$ stands for the sequence of Legendre polynomials: $%
\int_{-1}^1 P_{\ell}(t) P_{\ell^{\prime }}(t)\,dt= \frac{2}{2\ell+1}%
\delta_{\ell}^{\ell^{\prime }}, $ $\delta_{\ell}^{\ell^{\prime }}$ denoting
the Kronecker delta, see \cite[Section 4.7]{Sze75}. Note that the uniform
convergence of the series (\ref{eq2}) is equivalent to
\begin{equation*}
\sum_{\ell = 0}^{+\infty} \frac{2\ell+1}{4\pi} C_\ell(0) < +\infty
\end{equation*}
($C_\ell(0)\ge 0$ for every $\ell\ge 0$). Under Assumption \ref{regular},
the spatial derivatives up to order two of the covariance function admit a
series representation of the form (\ref{eq2}) with the Legendre polynomials
replaced by their derivatives: uniformly for $(\eta, \tau) \in [-1,1] \times
\mathbb{R}$
\begin{equation}  \label{eqDer}
\frac{\partial}{\partial \eta}\Gamma(\eta, \tau)= \sum_{\ell=0}^{+\infty}
\frac{2\ell+1}{4\pi} C_\ell(\tau) \frac{\partial}{\partial \eta} P_\ell(\eta
),\quad \frac{\partial^2}{\partial \eta^2}\Gamma(\eta, \tau) =
\sum_{\ell=0}^{+\infty} \frac{2\ell+1}{4\pi} C_\ell(\tau) \frac{\partial^2}{%
\partial \eta^2}P_\ell(\eta ).
\end{equation}
On the other hand, the uniform convergence of the series (\ref{eqDer}) is
equivalent to
\begin{equation}  \label{eqConv}
\sum_{\ell = 0}^{+\infty} \frac{2\ell+1}{4\pi} C_\ell(0) \cdot \ell^2 <
+\infty.
\end{equation}
An analogous result holds for time (and mixed) derivatives of $\Gamma$ still
up to order two but we omit the details since we will not need them in this
manuscript.

Let us now introduce some more notation: first denote by $\lbrace
Y_{\ell,m}, \ell \ge 0, m=-\ell, \dots, \ell\rbrace$ the standard real
orthonormal basis of spherical harmonics \cite[Section 3.4]{MaPeCUP} for $%
L^2(\mathbb{S}^2)$, then define for $\ell\in \mathbb{N}, m=-\ell,\dots,\ell$%
,
\begin{equation}  \label{eq3}
a_{\ell,m}(t) := \int_{\mathbb{S}^2} Z(x,t) Y_{\ell,m}(x)\,dx,\quad t\in
\mathbb{R}.
\end{equation}
From (\ref{eq3}) we deduce that $\lbrace a_{\ell,m}, \ell \ge 0, m=-\ell,
\dots, \ell\rbrace$ is a family of independent, stationary, centered,
Gaussian processes on the real line such that for every $t,s \in \mathbb{R}$
\begin{equation*}
\mathbb{E}[a_{\ell,m}(t) a_{\ell,m}(s) ] = C_\ell(t-s).
\end{equation*}
The spectral representation (\ref{eq2}) for $\Gamma$ allows to deduce the
so-called Karhunen-Lo\`eve expansion for $Z$:
\begin{equation}  \label{eqSeries}
Z(x,t) = \sum_{\ell =0}^{+\infty} \sum_{m=-\ell}^\ell a_{\ell,m}(t)
Y_{\ell,m}(x),
\end{equation}
where the stochastic processes $a_{\ell,m}$ are defined as in (\ref{eq3}),
and the series (\ref{eqSeries}) converges in $L^2(\Omega \times \mathbb{S}^2
\times [0,T])$ for any $T>0$. For $\mathbb{S}^2\ni x=(\theta_x, \varphi_x)$
we use the notation
\begin{equation*}
\partial_{1;x} = \frac{\partial}{\partial \theta}|_{\theta=\theta_x},\quad
\partial_{2;x}= \frac{1}{\sin\theta}\frac{\partial}{\partial}%
|_{\theta=\theta_x, \varphi=\varphi_x}.
\end{equation*}
Analogously, (\ref{eqDer}) ensures that, for $j=1,2$,
\begin{equation}  \label{eqSeries_gradient}
\partial_{j;x} Z(x,t) = \sum_{\ell =0}^{+\infty} \sum_{m=-\ell}^\ell
a_{\ell,m}(t) \partial_{j;x} Y_{\ell,m}(x),
\end{equation}
where the convergence still holds in $L^2(\Omega \times \mathbb{S}^2 \times
[0,T])$.

From now on we can restrict ourselves to
\begin{equation*}
\widetilde{\mathbb{N}} := \lbrace \ell \ge 0 : C_\ell(0)\ne 0\rbrace
\end{equation*}
without loss of generality.

\subsubsection{Time-dependent random spherical eigenfunctions}

\label{sec-rsh}

Let us define
\begin{equation}  \label{eqZell}
Z_\ell(x,t) := \sum_{m=-\ell}^\ell a_{\ell,m}(t) Y_{\ell,m}(x)\, , \quad
(x,t)\in \mathbb{S}^2\times \mathbb{R};
\end{equation}
by construction, $\lbrace Z_\ell, \ell\in \widetilde{\mathbb{N}}\rbrace$ is
a sequence of independent random fields and each $Z_\ell(\cdot, t)$ almost
surely solves the Helmholtz equation $\Delta Z_{\ell }(\cdot,
t)+\lambda_\ell Z_{\ell }(\cdot, t)=0, $ where $\Delta$ is the spherical
Laplacian and $\lambda_\ell:=\ell (\ell +1)$ is the $\ell$-th eigenvalue, in
coordinates
\begin{equation*}
\Delta = \frac{1}{\sin \theta} \frac{\partial}{\partial \theta} \left ( \sin
\theta \frac{\partial}{\partial \theta} \right ) + \frac{1}{\sin^2\theta}
\frac{\partial^2}{\partial \varphi^2}.
\end{equation*}
For the sake of notational simplicity we will assume that (cf. (\ref{eqConv}%
))
\begin{equation}  \label{somma1}
\sigma _{0}^{2}:=\mathbb{E}\left[ Z^{2}(x,t)\right] =\sum_{\ell\in
\widetilde{\mathbb{N}}} \mathbb{E}[Z_\ell(x,t)^2]=\sum_{\ell\in \widetilde{%
\mathbb{N}}}\frac{2\ell +1}{4\pi }C_{\ell }(0)=1.
\end{equation}


\textbf{Some conventions}. From now on, $c\in (0, +\infty)$ will stand for a
universal constant which may change from line to line. Let $\lbrace a_n,
n\ge 0\rbrace$, $\lbrace b_n, n\ge 0\rbrace$ be two sequences of positive
numbers: we will write $a_n \sim b_n$ if $a_n/b_n \to 1$ as $n\to +\infty$, $%
a_n \approx b_n$ whenever $a_n/b_n \to c$, $a_n = o(b_n)$ if $a_n/b_n \to 0$%
, and finally $a_n = O(b_n)$ or equivalently $a_n \ll b_n$ if eventually $%
a_n/b_n \le c$.

\subsection{Time dependence properties}



As in \cite{MRV:20}, let us now define the family of symmetric real-valued
functions $\lbrace g_\beta, \beta\in (0,1]\rbrace$ as follows:
\begin{equation}  \label{gbeta}
g_\beta(\tau) =
\begin{cases}
(1 + |\tau|)^{-\beta}\, & \text{if }\,\, \beta\in (0,1) \\
(1+|\tau|)^{-\alpha}\, & \text{if }\,\, \beta = 1%
\end{cases}%
\,,
\end{equation}
for some $\alpha \in [2,+\infty)$. We do believe that the assumption $\alpha
\in [2,+\infty)$ is not essential for the validity of our main findings;
indeed it seems likely that it can be replaced with $\alpha \in (1, +\infty)$%
. Nevertheless, the current formulation is instrumental to take advantage of
some technical results provided in \cite{MRV:20}.

\begin{condition}
\label{basic3} There exists a sequence $\lbrace \beta_\ell\in (0,1], \ell
\in \widetilde{\mathbb{N}}\rbrace$ such that
\begin{equation*}
C_\ell(\tau) = G_\ell(\tau) \cdot g_{\beta_\ell}(\tau), \qquad \ell\in
\widetilde{\mathbb{N}},
\end{equation*}
where $g_{\beta_\ell}$ is as in \eqref{gbeta} and
\begin{equation*}
\sup_{\ell\in \widetilde{\mathbb{N}}} \left | \frac{G_\ell(\tau)}{C_\ell(0)}
- 1\right | =o(1), \quad \text{as }\tau\to +\infty\,.
\end{equation*}
Moreover $0\in \widetilde{\mathbb{N}}$ (that is, $C_0(0)\ne 0$) and if $%
\beta_0=1$ then
\begin{equation*}
\int_{\mathbb{R}} C_0 (\tau)\, d\tau > 0\,.
\end{equation*}
\end{condition}

\emph{From now on we assume that Assumption \ref{basic3} holds for the
sequence $\lbrace C_\ell, \ell\in \widetilde{\mathbb{N}}\rbrace$.} Note that
$G_\ell(0)=C_\ell(0)$ for every $\ell\in \widetilde{\mathbb{N}}$.

Let $\ell\in \widetilde{\mathbb{N}}$. As discussed in \cite{MRV:20}, the
coefficient $\beta_{\ell }$ in Assumption \ref{basic3} governs the \emph{%
memory} of our processes; indeed, for $\beta_{\ell }=1$ (resp.~$%
\beta_\ell\in (0,1)$) the covariance function $C_{\ell }$ is integrable on $%
\mathbb{R}$ (resp.~$\int_{\mathbb{R}} |C_\ell(\tau)|\,d\tau = +\infty$) and
the corresponding process has so-called short (resp.~long) memory behavior
(note that $C_\ell(0)$ is always non-negative but $C_\ell(\tau)$ need not
be, for $\tau>0 $).


Clearly one could choose alternative parametrizations for $g_\beta(\tau)$,
such as for instance
\begin{equation*}
g_\beta(\tau) =(1 + |\tau|^2)^{-\beta/2},\quad \text{or} \quad g_\beta(\tau)
=(1 + |\tau|^{\gamma})^{-\beta};
\end{equation*}
however, these choices obviously cannot change our results, as our condition
is basically requiring that, for all $\ell$,
\begin{equation*}
\lim_{\tau\rightarrow \infty}\frac{C_\ell(\tau)}{C_\ell(0)\tau^{-\beta_\ell}}%
=1 .
\end{equation*}
A possible generalizations would be to allow for the possibility of
slowly-varying factors, i.e. to allow for autocorrelations of the form $%
L(|\tau|)\tau^{-\beta}$, where $L(\cdot )$ is such that $\lim_{\tau
\rightarrow \infty} L(|\tau|)/L(a|\tau|)=1$ for all $a>0$. These
generalizations are common in the long memory literature but would not alter
by any means the substance of our results, so we avoid to consider them for
brevity's sake.


\section{Main Results}

\label{sec-MR}

In this Section we introduce the problem and describe our main results. We
start with some technical lemmas.

\subsection{The average boundary length process}

\label{subsec-object}

Let $u\in \mathbb{R}$ be a threshold fixed from now on, for $t\in \mathbb{R}$
we consider the level set
\begin{equation*}
Z(\cdot, t)^{-1}(u):=\lbrace x\in \mathbb{S}^2 : Z(x,t) = u\rbrace
\end{equation*}
which is a.s. a $\mathcal{C}^1$ manifold of dimension $1$ thanks to
Bulinskaya's Lemma \cite[Proposition 1.20]{AW:09}. Indeed, for every $x\in
Z(\cdot, t)^{-1}(u)$ the random vector $(Z(x,t), \nabla Z(x,t))$ is
non-degenerate hence, except for a negligible subset of $\Omega$ which \emph{%
may depend on $t$}, the value $u$ is regular for $Z(\cdot, t)$, i.e. $\nabla
Z(x, t)\ne 0$ for every $x$ such that $Z(x,t)=u$. We are interested in
\begin{equation}  \label{length}
\mathcal{L}_u(t) := \text{length}(Z(\cdot, t)^{-1}(u)).
\end{equation}
By (time) stationarity of $Z$, the law of $\mathcal{L}_u(t)$ does not depend
on $t$, in particular $\mathbb{E}[\mathcal{L}_u(t)]$ does not depend on $t$,
and can be computed via the Kac-Rice formula \cite[Theorem 6.8]{AW:09} or
the Gaussian Kinematic Formula \cite[Theorem 13.2.1]{adlertaylor} to be
\begin{equation}  \label{meanGKF}
\mathbb{E}[\mathcal{L}_u(t)] = \sigma_1 \cdot 2\pi e^{-u^2/2},
\end{equation}
where for any $j=1,2$
\begin{equation}  \label{sigma1}
\sigma _{1}^{2}:=\mathbb{E}\left[ \partial _{j;x}Z(x,t)\,\partial
_{j;x}Z(x,t)\right] =\sum_{\ell =0}^{\infty}\frac{2\ell +1}{4\pi }C_{\ell
}(0)\frac{\ell (\ell +1)}{2},
\end{equation}
cf. (\ref{eqConv}). See Lemma \ref{lemcov} in Appendix \ref{app-cov} for
details on the covariance structure of the field $(Z, \nabla Z)$.

\begin{lemma}
\label{lem_basic} There exists $\tilde \Omega \subseteq \Omega$ such that $%
\mathbb{P}(\tilde \Omega)=1$ and for every $\omega\in \tilde \Omega$ there
exists $I(\omega)\subseteq [0,+\infty)$ whose complement is negligible such
that the value $u$ is regular for $Z(\cdot, t)(\omega)$ for every $t\in
I(\omega)$.
\end{lemma}

In view of Lemma \ref{lem_basic}, whose proof is postponed to the Appendix %
\ref{app-square}, we can define on $(\Omega, \mathfrak{F}, \mathbb{P})$ the
family $\lbrace \mathcal{C}_T(u), T>0\rbrace$ of random variables indexed by
$T>0$, where
\begin{equation}  \label{rap_int}
\mathcal{C}_{T}(u):=\,\int_{0}^{T}\,\Big(\mathcal{L}_{u}(t)-\mathbb{E}[%
\mathcal{L}_{u}(t)]\Big)\,dt,
\end{equation}
to be the \emph{average $u$-boundary length process} associated to $Z$. In
this paper we are interested in the behaviour of $\mathcal{C}_T(u)$, as $%
T\to +\infty$. To this purpose, in order to take advantage of Wiener-It\^o
theory, we first need to check that $\mathcal{C}_T(u)$ has finite variance.

\begin{lemma}
\label{squareL} For any $T>0$, the random variable $\mathcal{C}_T(u)$ is
square integrable, i.e., $\mathrm{Var}(\mathcal{C}_T(u)) < +\infty$.
\end{lemma}

In view of Lemma \ref{squareL}, whose proof is also postponed to the Appendix \ref%
{app-square}, $\mathcal{C}_T (u)$ in \eqref{rap_int} can be expanded into
so-called Wiener chaoses, by means of the Stroock-Varadhan decomposition,
see Section \ref{sec-Wiener} for details in our setting and \cite[\S 2.2]%
{noupebook} for a complete discussion. Briefly, this expansion is based on
the fact that the sequence of (normalized) Hermite polynomials $\lbrace H_q/%
\sqrt{q!}\rbrace_{q\ge 0}$
\begin{equation}  \label{Herm}
H_0\equiv 1,\qquad H_{q}(u):=(-1)^{q}\phi (u)^{-1}\frac{d^q}{du^q}\phi(u), \
q\ge 1
\end{equation}
(where $\phi$ denotes the probability density function of a standard
Gaussian random variable) is a complete orthonormal basis of the space of
square integrable functions on the real line with respect to the standard
Gaussian measure. (The first polynomials are $H_0(u)=1$, $H_1(u)=u$, $%
H_2(u)=u^2-1$, $H_3(u) = u^3 - 3u$.) We can write
\begin{equation}  \label{exp_s}
\mathcal{C}_{T}(u)=\sum_{q=0}^{\infty }\mathcal{C}_{T}(u)[q],
\end{equation}
where the series is orthogonal and converges in $L^2(\Omega)$, here $%
\mathcal{C}_{T}(u)[q]$ denotes the orthogonal projection of $\mathcal{C}%
_T(u) $ onto the so-called $q$-th Wiener chaos. In Proposition \ref{teoexpS}
we will determine analytic formulas for these chaotic components. We will
exploit the series representation (\ref{exp_s}) to investigate the
asymptotic distribution of $\mathcal{C}_T(u)$ as $T\to +\infty$. Roughly
speaking, in the long memory regime the behavior of $\mathcal{C}_T(u)$ will
be determined by a single term of the series, while in the case of short
range dependence all chaotic components will contribute in the limit thus
influencing both the asymptotic variance and the nature of second order
fluctuations of our boundary length functional.

\begin{remark}
This Hermite-type expansion can be directly given for $\mathcal{L}%
_u(t)$ in (\ref{length}), thus being also instrumental for the relations
between different geometric processes (evolving over time) associated to the
random field $Z$, such as the area of excursion sets and their
Euler-Poincar\'e characteristic, which we plan to investigate in a future
paper.
\end{remark}

\subsection{Statement of main results}

Before stating our main results we need some more notation.

\begin{condition}
\label{condbeta} Let $\lbrace \beta_\ell, \ell \in \widetilde{\mathbb{N}}
\rbrace$ be the sequence defined in Assumption \ref{basic3}.

\begin{itemize}
\item The sequence $\lbrace \beta_\ell, \ell \in \widetilde{\mathbb{N}},
\ell\ge 1\rbrace$ admits minimum. Let us set
\begin{equation}  \label{def_star}
\beta_{\ell ^{\star }}:=\min\lbrace \beta _{\ell }, \ell \in \widetilde{%
\mathbb{N}},\ell\ge 1\rbrace,\qquad \mathcal{I}^{\star }:=\{\ell\in
\widetilde{\mathbb{N}} :\beta _{\ell }=\beta _{\ell ^{\star }}\}.
\end{equation}

\item If $\mathcal{I}^\star\ne \widetilde{\mathbb{N}}$, then the sequence $%
\lbrace \beta_\ell, \ell\in \widetilde{\mathbb{N}} \setminus \mathcal{I}%
^\star,\ell\ge 1\rbrace$ admits minimum. Let us set
\begin{equation}  \label{def_starstar}
\beta_{\ell ^{\star \star }}:=\min \left\{ \beta _{\ell },\ \ell \in \mathbb{%
N}\backslash \mathcal{I}^{\star },\ell\ge 1\right\}.
\end{equation}
\end{itemize}
\end{condition}

Note that $\beta_{\ell^\star}, \beta_{\ell^{\star\star}}\in (0,1]$ and for $%
\ell\in \mathcal{I}^\star$, obviously $C_{\ell}(0)> 0$. In words, $\beta
_{\ell ^{\star }}$ represents the smallest exponent corresponding to the
largest memory, $\mathcal{I}^{\star }$ the set of multipoles where this
minimum is achieved, and $\beta _{\ell ^{\star \star }}$ the second smallest
exponent $\beta _{\ell }$ governing the time decay of the autocovariance $%
C_\ell$ at some given multipole $\ell$. Note that we are \textit{excluding}
the multipole $\ell =0$ by the definition of $\beta_{\ell^\star}$ and $%
\beta_{\ell^{\star\star}}$ in \eqref{def_star} and \eqref{def_starstar}, on
the other hand $\ell=0$ may belong to $\mathcal{I}^\star$. \emph{From now on
we work under Assumption \ref{condbeta}.}


\subsubsection{Long range dependence}

As briefly anticipated above, for long memory random fields a single chaotic
component determines the asymptotic behavior of $\mathcal{C}_T(u)$. In our
setting, the role of dominating term is played by $\mathcal{C}_T(u)[2]$ ((%
\ref{exp_s}) with $q=2$): we will see in Remark \ref{2chaos-hermite} that
\begin{equation}  \label{caos2}
\mathcal{C}_T(u)[2]=\frac{\sigma _{1}}{2}\sqrt{\frac{\pi }{2}}\phi
(u)\sum_{\ell } \frac{C_{\ell }(0)(2\ell +1)}{4\pi } \left\{ (u^{2}-1)+\frac{%
\lambda _{\ell }/2}{\sigma _{1}^{2}}\right\} \int_0^T\int_{\mathbb{S}%
^{2}}H_{2}(\widehat Z_{\ell }(x,t))dxdt,
\end{equation}
{where $\lambda_\ell$ still denotes the $\ell$-th eigenvalue of the
spherical Laplacian, $\sigma_1$ is defined as in (\ref{sigma1}), $%
H_2(t)=t^2-1$ denotes the second Hermite polynomial, and $\widehat{Z}_\ell$
is defined as
\begin{equation}  \label{zhat}
\widehat Z_\ell (x,t) := \frac{Z_\ell(x,t)}{\sqrt{\frac{2\ell+1}{4\pi}%
C_\ell(0)}}, \quad (x,t)\in \mathbb{S}^2\times \mathbb{R}
\end{equation}
recalling the content of Section \ref{sec-rsh}. In particular, the $\widehat
Z_\ell$'s are unit variance time-dependent random spherical harmonics. The
asymptotic law of $\mathcal{C}_T(u)[2]$ was introduced in \cite{MRV:20} and
it is related to the Rosenblatt distribution, see below. }

\begin{definition}
The random variable $X_\beta$ has the standard Rosenblatt distribution (see
e.g. \cite{Ta:75} and also ~\cite{DM79,Ta:79}) with parameter $\beta \in (0,%
\frac{1}{2})$ if it can be written as
\begin{equation}  \label{Xbeta}
X_{\beta}=a(\beta) \int_{(\mathbb{R}^{2})^{^{\prime }}}\frac{e^{i(\lambda
_{1}+\lambda _{2})}-1}{i(\lambda _{1}+\lambda _{2})}\frac{W(d\lambda
_{1})W(d\lambda _{2})}{|\lambda _{1}\lambda _{2}|^{(1-\beta )/2}}\text{ ,}
\end{equation}%
where $W$ is the white noise Gaussian measure on $\mathbb{R}$, the
stochastic integral is defined in the Ito's sense (excluding the diagonals:
as usual, $(\mathbb{R}^2)^{\prime }$ stands for the set $\lbrace (\lambda_1,
\lambda_2)\in \mathbb{R}^2:\lambda_1 \ne \lambda_2\rbrace$), and
\begin{equation}  \label{abeta}
a(\beta) := \frac{\sigma(\beta)}{2\,\Gamma(\beta)\,\sin\left({(1-\beta)\pi}/{%
2}\right)}\,,
\end{equation}
with
\begin{equation*}
\sigma(\beta) := \sqrt{\frac12(1-2\beta)(1-\beta)}\,.
\end{equation*}
Following \cite{MRV:20}, we say the random vector $V$ satisfies a composite
Rosenblatt distribution of degree $N\in \mathbb{N}$ with parameters $%
c_{1},...,c_{N}\in \mathbb{R},$ if%
\begin{equation}  \label{V}
V=V_{N}(c_{1},...,c_{N};\beta )\mathop{=}^d\sum_{k=1}^{N}c_{k}X_{k;\beta}%
\text{ ,}
\end{equation}%
where $\left\{ X_{k;\beta}\right\} _{k=1,...,N}$ is a collection of i.i.d.
standard Rosenblatt random variables of parameter $\beta $.
\end{definition}

Note that indeed $\mathbb{E}[X_\beta]=0$ and $\mathop{\rm Var}(X_\beta)=1$.
The Rosenblatt distribution was first introduced in \cite{Ta:75} and has
already appeared in the context of spherical isotropic Gaussian random
fields as the exact distribution of the correlogramm, see \cite{LTT18}.


Further characterizations of the composite Rosenblatt distribution, for
instance in terms of its characteristic function, can be found in \cite%
{MRV:20}.


We are now ready to state our first main result. Let us define the
standardized average boundary length functional as
\begin{equation}
\widetilde{\mathcal{C}}_{T}(u):=\frac{\mathcal{C}_{T}(u)}{\sqrt{\mathop{\rm
Var}\left ( \mathcal{C}_{T}(u)\right ) }}.
\end{equation}

\begin{theorem}
\label{main-theorem} If $2\beta_{\ell^\star}<\min(\beta_0,1)$, then as $T\to
+\infty$,
\begin{equation}  \label{caos2dom}
\widetilde{\mathcal{C}}_{T}(u) =\frac{\mathcal{C}_{T}(u)[2]}{\sqrt{%
\mathop{\rm Var}\left ( \mathcal{C}_{T}(u)\right ) }}+o_{\mathbb{P}}(1)\,,
\end{equation}
where $o_{\mathbb{P}}(1)$ denotes a sequence converging to zero in
probability, $\mathop{\rm Var}\left ( \mathcal{C}_{T}(u)\right ) \sim %
\mathop{\rm Var}\left ( \mathcal{C}_{T}(u)[2]\right )$ and
\begin{equation}  \label{asymp_var}
\lim_{T\rightarrow \infty }\frac{\mathop{\rm Var}\left( \mathcal{C}%
_{T}(u)[2]\right) }{T^{2-2\beta _{\ell^{\star}}}} =\frac{\sigma _{1}^{2}\pi%
} {4}\phi ^{2}(u)\sum_{\ell \in \mathcal{I}^\star}\frac{(2\ell+1)^{2}C_%
\ell(0)^2}{(1-2\beta_\ell)(1-\beta_\ell)}\left\{ (u^{2}-1)+\frac{\lambda
_{\ell}/2}{\sigma_{1}^{2}}\right\}^2\,.
\end{equation}
Assume in addition that $\# \mathcal{I}^\star$ in (\ref{def_star}) is
finite, then as $T\to \infty $, we have that
\begin{equation*}
\widetilde{\mathcal{C}}_{T}(u)\overset{d}{\longrightarrow}\sum_{\ell \in
\mathcal{I}^\star} \frac{C_\ell(0)}{\sqrt{v^\star}}\left\{ (u^{2}-1)+\frac{%
\lambda _{\ell}/2}{\sigma_{1}^{2}}\right\}V_{2\ell+1}(1,\dots,1;\beta_{\ell^%
\star})\,,
\end{equation*}
where
\begin{equation*}
v^\star=a(\beta_{\ell^\star})^2 \sum_{\ell \in \mathcal{I}^\star} \frac{%
2(2\ell+1)^{2} C_\ell(0)^2}{(1-2\beta_\ell)(1-\beta_\ell)} \left\{ (u^{2}-1)+%
\frac{\lambda _{\ell}/2}{\sigma_{1}^{2}}\right\}^2\,,
\end{equation*}
$a(\beta_{\ell^\star})$ being as in \eqref{abeta}.
\end{theorem}

To investigate more deeply the structure of the dominating limit variables,
we can distinguish between two cases:

\begin{enumerate}
\item \textbf{Non-unique minimum.} This is the situation where
at least one multipole has non integrable (over time) autocovariance
function, but the cardinality of $\mathcal{I}^{\star }$ is strictly larger
than one, $\#\mathcal{I}^{\star }>1$, meaning that the minimum of $\lbrace
\beta _{\ell }, \ell\in \widetilde{\mathbb{N}} \rbrace$ is non-unique. In
this case the dominating second order chaos has a neat expression for the
variance but the Berry cancellation phenomenon cannot occur, i.e., the
variance has the same order of magnitude at any level $u\in\mathbb{R}$, see (%
\ref{asymp_var}). 

\item \textbf{Unique minimum}. This is the situation where at least one
multipole has non integrable (over time) autocovariance function and $\#%
\mathcal{I}^{\star }=1,$ meaning that there is a single multipole (labelled $%
\ell ^{\star })$ where $\lbrace \beta _{\ell }, \ell\in \widetilde{\mathbb{N}%
} \rbrace$ achieves its minimum, and hence where the temporal dependence is
maximal. In these circumstances, we have not only that the boundary length
is dominated by the second chaos (\ref{caos2dom}), but also that this chaos
admits an asymptotic expression in terms of the (random) $L^{2}(\mathbb{S}%
^2) $-norm of the field $Z_{\ell^{\star}}$: from (\ref{caos2})
\begin{equation}  \label{2nd-mono}
\mathcal{C}_T(u)[2]=\frac{\sigma _{1}}{2}\sqrt{\frac{\pi }{2}}\phi (u)(2\ell
^{\star }+1)\left\{ (u^{2}-1)+\frac{\lambda _{\ell ^{\star }}/2}{\sigma
_{1}^{2}}\right\} \frac{C_{\ell^\star}(0)}{4\pi } \int_{\mathbb{S}%
^{2}}H_{2}(\widehat Z_{\ell^\star}(x,t))dxdt.
\end{equation}
Moreover, from (\ref{2nd-mono}) \emph{perfect correlation occurs between the
average boundary length and the average excursion area investigated in \cite%
{MRV:20}}. In this case the variance is asymptotic to
\begin{equation*}
\mathop{\rm Var}\left( \mathcal{C}_{T}(u)\right)\sim{T^{2-2\beta_{\ell^%
\star}}}\frac{\sigma_{1}^{2}\pi \phi ^{2}(u)}{4} \frac{(2\ell^\star+1)^{2}
C_{\ell^\star}(0)^2}{(1-2\beta_{\ell^\star})(1-\beta_{\ell^\star})} \left\{
(u^{2}-1)+\frac{\lambda_{\ell^\star}/2}{\sigma_{1}^{2}}\right\}^2
\end{equation*}
and a form of Berry's cancellation phenomenon holds, meaning that the
variance of boundary lengths at some levels (not zero) has a smaller order
of magnitude than the variance at any other level. Indeed, Berry's
cancellation occurs in points $u$ such that
\begin{equation*}
(u^{2}-1)+\frac{\lambda _{\ell }/2}{\sigma _{1}^{2}}=0,\quad\text{ i.e. }%
u^{2}=1-\frac{\lambda _{\ell ^{\star }}/2}{\sigma _{1}^{2}}\text{ .}
\end{equation*}%
This clearly implies that Berry's cancellation can occur in a point $\pm
u^{\star }\in \lbrack 0,1)$ such that%
\begin{equation*}
u^{\star }=\pm \sqrt{1-\frac{\mathbb{\lambda }_{\ell ^{\star }}}{\mathbb{E}%
_{C}[\mathbb{\lambda }_{\ell }]}}
\end{equation*}%
where $\mathbb{E}_{C}[\mathbb{\lambda }_{\ell }]$ is the expected value
under the probability measure assigning weights $\frac{2\ell +1}{4\pi }%
C_{\ell }(0)$ (recall that $\sigma_0^2=\sum \frac{2\ell +1}{4\pi }C_{\ell
}(0)=1$). In other words, for Berry's cancellation to occur we need long
memory to occur in multipoles which are ``lower than average" in terms of
the angular power spectrum, or we need the field at any given time to have
huge power on low multipoles. Indeed in the standard monochromatic wave case
we have $\frac{\mathbb{\lambda }_{\ell ^{\star }}}{\mathbb{E}_{C}[\mathbb{%
\lambda }_{\ell }]}=1$ and we are back to the nodal length case. At $u^\star$
we have two possible scenarios: if $2\beta_{\ell^{\star\star}}>3\beta_{%
\ell^{\star}}$ then the boundary length is dominated by the third chaos,
while if $2\beta_{\ell^{\star\star}}<3\beta_{\ell^{\star}}$ then the
boundary length is still dominated by its second chaotic component.
\end{enumerate}

\begin{remark}
\textrm{More generally, a crucial role seems to be played by the dispersion
of the ``random" quantity $\frac{\lambda _{\ell }/2}{\sigma _{1}^{2}}$ in
terms of the ``weights"
\begin{equation*}
W_{\ell }(T):=\mathrm{Var}\left\{ (2\ell +1)\int_{0}^{T}\left[ \widehat{C}%
_{\ell }(t)-\mathbb{E}\widehat{C}_{\ell }(t)\right] dt\right\} \text{ ,}
\end{equation*}%
where $\widehat{C}_\ell$ is the sample spectrum, defined as
\begin{equation}  \label{sampletris}
\widehat{C}_{\ell }(t):=\frac{1}{2\ell +1}\sum_{m=-\ell }^{\ell }\left\vert
a_{\ell m}(t)\right\vert ^{2}=\frac{1}{2\ell +1}\int_{\mathbb{S}^{2}}Z_{\ell
}(x,t)^{2}dx\text{ .}
\end{equation}
In particular, the correlation with the excursion area is going to be larger
and larger as the ``variance" of $\frac{\lambda _{\ell }/2}{\sigma _{1}^{2}}$
is going to be smaller and smaller. }
\end{remark}

\subsubsection{Short Range Dependence}

If the set of long range dependent multipoles $\mathcal{I}$ is empty,
meaning that $\beta_0=1$ and $2\beta _{\ell }>1$ for all $\ell\ge1$, then
the second-order chaotic component would no longer be dominating, and
investigation of all terms of the series (\ref{chaosexpS}) below is
required. In this case a Gaussian limit via classic Breuer-Major arguments
\cite{BM} holds.

We first need to introduce some more notation: for $q\ge 1$, let
\begin{equation*}
s^2_q:=\lim_{T \rightarrow \infty}\frac{Var[\mathcal{C}_T(u)[q]]}{T};
\end{equation*}



\begin{theorem}
\label{quattro} Assume $\beta_0=1$ and $2\beta _{\ell }>1$ for all $\ell\ge1$%
. Then we have

\begin{equation*}
\lim_{T\rightarrow \infty }\frac{\mathrm{Var}\left( \mathcal{C}%
_{T}(u)\right) }{T}=\sum_{q=1}^{+\infty} s_q^2\, ,
\end{equation*}
and moreover, as $T\to +\infty$,
\begin{equation*}
\widetilde{\mathcal{C}}_{T}(u)=\frac{\mathcal{C}_{T}(u)-\mathbb{E}{\mathcal{C%
}}_{T}(u)}{\sqrt{Var[\mathcal{C}}_{T}(u)]}\overset{d}{\rightarrow} Z,
\end{equation*}
$Z\sim \mathcal{N}(0,1)$ being a standard Gaussian random variable.
\end{theorem}

Recall that for $\beta_0=1$ we have $\int_{\mathbb{R}} C_0(\tau)\,d\tau \in
(0,+\infty)$ (see Condition \ref{basic3}) so that $s_1^2>0$ yielding $%
\sum_{q\ge 1}s^2_q >0$ (the limiting variance constant is strictly
positive). The proof of the previous result can then be established by a
standard (although lengthy) analysis of terms in the chaos expansions (\ref%
{chaosexpS}). In particular, note that by the $L^2$ convergence of the
Wiener chaoses it is sufficient to focus on an (arbitrarily large but)
finite number of components (the remainder may be made negligible, uniformly
over $T$); the fourth cumulants of these components con be shown to converge
to zero after normalizing for the variance, so that the Central Limit
Theorem may follow from Stein-Malliavin arguments (see \cite{noupebook}).
Details are omitted for brevity's sake.


\subsection{Structure of the paper}

In Section \ref{sec-discussion} we compare our main findings with the
existing literature. Section \ref{secProofsMR} contains the proof of our
main theorem, together with the presentation of its main technical tool and
it is divided as follows. In Section 4.1 we present the $L^2$ approximation
of the length for level curves, whereas the Wiener chaotic decomposition of
our boundary length functional is given in Section 4.2; in particular we
study the second order chaotic component, i.e. we compute its variance and
hence we obtain a much neater asymptotic expression, which includes only the
multipoles corresponding to the strongest memory. A much more technical
computation is aimed to show that all the higher-order chaotic components
are asymptotically negligible; these results are then combined in Section
4.2 to prove our main theorem. The Appendix collects a number of important
auxiliary results, that derive explicitly covariance structures and cover
measurability issues, mean-square approximations and chaotic decompositions.

\subsection{Acknowledgements}

DM acknowledges the MIUR Excellence Department Project awarded to the
Department of Mathematics, University of Rome “Tor Vergata%
”CUP E83C18000100006. The research of MR has been supported by
the ANR-17-CE40-0008 Project UNIRANDOM. AV has been supported by the
co-financing of the European Union - FSE-REACT-EU, PON Research and
Innovation 2014-2020, DM 1062/2021.

\section{Discussion}

\label{sec-discussion}

In this Section we compare our main results with the existing related
literature on the geometry of random fields.

\subsection{A comparison with the high-energy regime literature}

The literature on the geometry of random fields on manifolds has become vast
over the last decade, see for instance \cite{Rossi2018}, \cite{M2022}, \cite%
{Wig22} for some recent surveys. Much of the literature has concentrated on
the high-frequency geometry for random eigenfunctions, in the case of random
fields on the sphere (or on other Riemannian manifolds, for instance the
torus) with no temporal dependence. In particular, concerning level curves
it has been shown that the following asymptotic results hold:

\begin{itemize}
\item for level sets corresponding to $u\neq 0,$ the length of level curves
is dominated by a single projection term in the chaos expansion, i.e., the
second-order component;

\item this component can be expressed in terms of the norm of the function,
without its derivatives and it disappears in the nodal case $u=0$ (the
so-called Berry's cancellation phenomenon);

\item at $u=0,$ the nodal length is again dominated by a single term in the
chaos expansion, which is the fourth-order component;

\item in both cases, it is possible to establish quantitative central limit
theorems, in the high-energy limit;

\item the nodal length and the level curves are asymptotically perfectly
uncorrelated. However, considering the partial autocorrelation, i.e.,
removing (or \emph{freezing}) the effect of the random $L^{2}$-norm, the
asymptotic correlation is again unity.
\end{itemize}

Much of these results can be extended to other geometric functionals, such
as Lipschitz-Killing curvatures (in the two-dimensional case, the excursion
area, the boundary length and the Euler-Poincar\'e characteristic) and
critical points. Indeed, full correlation has been shown to hold, in the
high-frequency limit, for all these statistics, in \emph{generic} cases
where $u\neq 0$ (or, more generally, where the second-order chaos component
does not disappear).

The setting we consider in this paper is rather different, for a number of
reasons. Firstly, we are not considering eigenfunctions, but arbitrary
(although Gaussian and isotropic) spherical random fields. More importantly,
we are going beyond those previous results by allowing a form of dependence
over time; because of this, the asymptotic framework of our work here is
based upon a different asymptotic regime, that is, fluctuations for a
growing span over time, rather than for higher and higher frequency
eigenfunctions. The results presented here show then both analogies and
important differences with the existing literature. More precisely, let us
note the following:

\begin{itemize}
\item It is still the case (in the long memory case) that the fluctuations
around the expected value are asymptotically (as $T\rightarrow \infty $)
dominated by the second order chaos; on one hand, this chaos can again be
expressed in terms of the harmonic components of the fields itself, without
the need to resort to derivatives (despite the fact that these derivatives
do appear in the Kac-Rice representation of level curves, see below).

\item On the other hand, it is no longer the case that the second-order
chaos is proportional to the random $L^{2}$-norm of the field itself; it is
instead a linear combination of $L^{2}$-norm of the harmonic components $%
Z_{\ell }$.

\item Related to the previous point, it is no longer the case that Berry's
cancellation occurs at the nodal level $u=0$, and actually for the case $\#%
\mathcal{I}^\star>1$, the Berry cancellation phenomenon cannot occur at all,
i.e., the variance has the same order of magnitude at any level $u\in\mathbb{%
R}$.

\item In the case of asymptotically monochromatic fields, i.e., those where
the minimum of the memory parameter $\beta_\ell$ is attained on a single
multipole, there exist levels where the second-order chaos disappears, and
hence the variance is asymptotically of lower order. The exact value of
these levels depends upon a combination of the variance of the single
component $Z_{\ell^\star}$, and the variance of the derivative of the entire
field $Z$. However, rather differently from the literature so far, these do
not correspond to the nodal case $u=0$.
\end{itemize}

Let us also recall that in \cite{MRV:20} the large time behavior of the
empirical excursion area of the space-time spherical random field $Z$ in (%
\ref{eq1}) has been investigated, i.e. the asymptotic distribution of
\begin{equation}  \label{area1}
\mathcal{M}_T(u) := \int_0^T \left ( \int_{\mathbb{S}^2}\left (1_{\lbrace
Z(x,t)\ge u\rbrace} - \mathbb{P}(Z(x,t)\ge u)\right)\,dx\right )\,dt
\end{equation}
as $T\to +\infty$. First of all it is worth mentioning that the analysis of (%
\ref{area1}) can be carried out under the sole Assumption \ref{basic}, for
the length of level curves instead we need more regularity for $Z$, as
explained at the beginning of Section \ref{sec-MR}. Moreover, for the
excursion area the zero-level $u=0$ is still a cancellation point under long
memory circumstances, while this is not the case for the variance of level
curves. More importantly, for the excursion area the second-order chaos is
proportional to the random $L^2$ norm of the random field, whereas for level
curves the second-order chaos is proportional to a linear combination of the
$L^2$ norms of the eigenfunctions of the field; the two chaoses are hence
not perfectly correlated, unless the fields are asymptotically monochromatic.

\section{Proofs of the main results}

\label{secProofsMR}

In this Section we prove our main results. As anticipated in Section \ref%
{sec-MR}, the starting point of our argument is the Stroock-Varadhan
decomposition of $\mathcal{C}_T(u)$ in (\ref{rap_int}).

\subsection{The $L^2$-approximation}

Let $t\in \mathbb{R}$, the length of $u$-level curves can be \emph{formally}
represented as
\begin{equation*}
\mathcal{L}_u(t) = \int_{\mathbb{S}^2} \delta_u(Z(x,t)) \| \nabla
Z(x,t)\|\,dx,
\end{equation*}
where $\delta_u$ is the Dirac mass in $u$. For $\epsilon >0$ consider the $%
\epsilon$-approximating $u$-level curves length (see Lemma \ref{approx2} in
Appendix \ref{app-square})
\begin{equation}  \label{Leps}
\mathcal{L}_u^\epsilon(t) := \frac{1}{2\epsilon}\int_{\mathbb{S}^2}
1_{[u-\epsilon,u+ \epsilon]}(Z(x,t)) \| \nabla Z(x,t)\|\,dx
\end{equation}
and define accordingly the $\epsilon$-approximating random variable
\begin{equation}  \label{Ceps}
\mathcal{C}^\epsilon_T(u) := \int_0^T \left ( \mathcal{L}_u^\epsilon(t) -
\mathbb{E}[ \mathcal{L}_u^\epsilon(t)]\right )\,dt.
\end{equation}
The following technical result is crucial and will be proved in Appendix \ref%
{app-square}.

\begin{lemma}
\label{approx} As $\epsilon\to 0$,
\begin{equation}
\mathcal{C}_T^\epsilon(u) \to \mathcal{C}_T(u)
\end{equation}
both a.s. and in $L^2(\mathbb{P})$, where $\mathcal{C}_T(u)$ is defined as
in (\ref{rap_int}).
\end{lemma}

\subsection{The Wiener chaos expansion}

\label{sec-Wiener}

In order to derive the analytic form for \eqref{exp_s} we get inspired by
the chaotic decomposition for level curves of Gaussian random fields found
in e.g. \cite{KL01, MPRW16}. Let us introduce the collection of coefficients
$\{\alpha_{n,m} : n,m\geq 1\}$ and $\{\beta_{l}(u) : l\geq 0\}$, related to
the (formal) Hermite expansions of the norm $\| \cdot \|$ in $\mathbb{R}^2$
and the Dirac mass $\delta_u(\cdot)$ respectively:
\begin{equation}  \label{e:beta}
\beta_{q}(u):= \phi(u)H_{q}(u),
\end{equation}
where $\phi$ is the standard Gaussian probability density function, $H_{l}$
denotes the $l$-th Hermite polynomial and $\alpha_{n,m}:=0$ but for the case
$n,m$ even
\begin{equation}  \label{e:alpha}
\alpha_{2n,2m}:=\sqrt{\frac{\pi}{2}}\frac{(2n)!(2m)!}{n! m!}\frac{1}{2^{n+m}}
p_{n+m}\left (\frac14 \right),
\end{equation}
where for $N=0, 1, 2, \dots $ and $x\in \mathbb{R}$
\begin{equation}  \label{pN}
p_{N}(x) :=\sum_{j=0}^{N}(-1)^{j}\ \ (-1)^{N}{\binom{N }{j}}\ \ \frac{(2j+1)!%
}{(j!)^2} x^j.
\end{equation}

%
In view of (\ref{sigma1}), we define the normalized gradient and derivatives
for $j=1,2$
\begin{equation*}
\widetilde \nabla := \nabla /\sigma_1,\quad \widetilde \partial_{j;x} :=
\partial_{j;x}/\sigma_1.
\end{equation*}

\begin{proposition}[Chaotic expansion for $\mathcal{C}_T(u)$]
\label{teoexpS} For every $T>0$ and $q\geq 1$,
\begin{flalign}  \label{e:ppS}
\mathcal{C}_T(u)[q]
&= \sigma_1 \,\sum_{m=0}^{q}%
\sum_{k=0}^{m} \frac{\alpha _{k,m-k}\beta_{q-m}(u) }{(k)!(m-k)!(q-m)!} \nonumber \\
& \quad \quad \times \int_0^T \int_{\mathbb{S}^2} H_{q-m}(Z(x,t))
H_{k}(\widetilde\partial_{1;x} Z(x,t))H_{m-k}(\widetilde\partial_{2;x}
Z(x,t))\,dx\,dt\,.
\end{flalign}
As a consequence, one has the representation
\begin{flalign}  \label{chaosexpS}
\mathcal{C}_T(u) &=\sigma_1\,\sum_{q=1}^{+\infty}%
\sum_{m=0}^{q}\sum_{k=0}^{m} \frac{\alpha _{k,m-k}\beta _{q-m}(u)}{%
(k)!(m-k)!(q-m)!} \, \times \nonumber \\
&\quad \times\, \int_0^T\int_{\mathbb{S}^2}\!\!
H_{q-m}(Z(x,t))H_{k}(\widetilde\partial_{1;x} Z(x,t))H_{m-k}(\widetilde{\partial}%
_{2;x} Z(x,t))\,dx\,dt\,,
\end{flalign}
where the series converges in $L^2(\Omega)$.
\end{proposition}

The proof of Proposition \ref{teoexpS} is postponed to the Appendix \ref{app-chaos}:
first we compute the chaotic expansion of $\mathcal{C}^\epsilon_T(u)$ in (%
\ref{Ceps}), then we let $\epsilon\to 0$ obtaining (\ref{chaosexpS}) thanks
to Lemma \ref{approx}.

Let us investigate the chaotic components (\ref{e:ppS}) starting from the
case $q=1$.

\subsubsection{The first chaotic projection}

The first term in the series (\ref{chaosexpS}) is
\begin{eqnarray}
\mathcal{C}_T(u)[1] &=& \sigma_1 \sqrt{\frac{\pi}{2}} u\phi(u) \int_{\mathbb{%
S}^2}Y_{0,0}(x)\,dx \int_0^T a_{0,0}(t)\,dt  \notag \\
&=& \sigma_1 \sqrt 2 \pi u\phi(u) \int_0^T a_{0,0}(t)\,dt.
\end{eqnarray}

\begin{lemma}
\label{lemvar1} We have, as $T\to +\infty$,
\begin{equation*}
\lim_{T\to\infty} \frac{\mathop{\rm Var}(\mathcal{C}_T(u)[1])}{T^{2-\beta_0}}
= \sigma_1^2 2\pi^2 u^2 \phi(u)^2\frac{2C_0(0)}{(1-\beta_0 )(2-\beta_0 )}%
,\quad \text{if } \beta_0\in (0,1)
\end{equation*}
and
\begin{equation*}
\lim_{T\to\infty}\frac{\mathop{\rm Var}(\mathcal{C}_T(u)[1])}{T} =
\sigma_1^2 2\pi^2 u^2\phi(u)^2 \int_{\mathbb{R}} C_0(\tau)\,d\tau,\quad
\text{if }\beta_0=1.
\end{equation*}
\end{lemma}

The proof of Lemma \ref{lemvar1} is identical to the proof of Lemma 4.2 in
\cite{MRV:20} and hence omitted. Recall that Assumption \ref{basic3} ensures
that $C_0(0)>0$ and that for $\beta_0=1$
\begin{equation*}
\int_{\mathbb{R}} C_0(\tau)\,d\tau \in (0,+\infty);
\end{equation*}
it is worth noticing that $\mathcal{C}_T(u)[1]\ne 0$ if and only if $u\ne 0$.

\subsubsection{The second order chaotic projection}

Recall the notation ($\lambda _{\ell }:=\ell(\ell +1)$) from (\ref{somma1})
and (\ref{sigma1})
\begin{equation*}
\sigma _{0}^{2} =\sum_{\ell }\frac{(2\ell +1)}{4\pi }C_{\ell }(0)=1,\quad
\sigma _{1}^{2} =\sum_{\ell }\frac{(2\ell +1)}{4\pi }C_{\ell }(0)\frac{%
\lambda _{\ell }}{2}.
\end{equation*}

\begin{proposition}
\label{propcaos2} The second order chaotic component can be written as
\begin{flalign}\label{2chaos}
\mathcal{C}_T(u)[2]&=\frac{\sigma _{1}}{2}\sqrt{\frac{\pi }{2}}\phi (u)\sum_{\ell
}(2\ell +1)\left\{ (u^{2}-1)+\frac{\lambda _{\ell }/2}{\sigma _{1}^{2}}%
\right\} \int_{[0,T]} \left\{ \widehat{C}_{\ell }(t)-C_{\ell }(0)\right\} dt \,,
\end{flalign}
where $\widehat{C}_{\ell }(t)$ is the sample power spectrum (\ref{sampletris}%
)
\begin{equation*}
\widehat{C}_{\ell }(t)=\frac{1}{2\ell +1}\sum_{m=-\ell }^{\ell }\left\vert
a_{\ell m}(t)\right\vert ^{2}=\frac{1}{2\ell +1}\int_{\mathbb{S}^{2}}Z_{\ell
}(x,t)^{2}dx\text{ .}
\end{equation*}
\end{proposition}

Note that for every $t\in \mathbb{R}$
\begin{equation*}
\mathbb{E}[\widehat C_\ell(t)] = C_\ell(0).
\end{equation*}
\proof[Proof of Proposition \ref{propcaos2}] From Proposition \ref{teoexpS},
\begin{flalign*}
\mathcal{C}_T(u)[2]
&= \sigma_1 \,\frac{\alpha _{0,0}\beta _{2}(u)}{2}
\,\int_0^T \int_{\mathbb{S}^2} H_{2}(Z(x,t))\,dx\,dt \\
&\quad+\sigma_1 \,\frac{\alpha _{2,0}\beta _{0}(u)}{2} \, \int_0^T\int_{\mathbb{S}^2}
\left ( \langle \widetilde \nabla Z(x,t), \widetilde \nabla Z(x,t)\rangle - 2 \right )\,dx\,dt.
\end{flalign*}
Recall the basic (Green-Stokes) identity for a regular function $T:\mathbb{S}%
^2\to \mathbb{R}$
\begin{equation*}
\int_{\mathbb{S}^{2}} \langle \nabla T, \nabla T \rangle dx=-\int_{\mathbb{S}%
^{2}}T\Delta T\,dx,
\end{equation*}%
then
\begin{flalign*}
\mathcal{C}_T(u)[2] &= \sigma_1 \frac{\alpha _{00}\beta _{2}(u)}{2}\int_{[0,T]}\int_{\mathbb{S}%
^{2}} \left(Z(x,t)^{2}-1\right ) dxdt
\\
&\quad -\sigma_1 \frac{\alpha _{20}\beta _{0}(u)}{2}\frac{1}{\sigma_1^2}\int_{[0,T]}\int_{%
\mathbb{S}^{2}} \left ( Z(x,t) \Delta Z(x,t) - 2\sigma_1^2\right ) d
x dt\\
&=\sigma_1 \frac{\alpha _{00}\beta _{2}(u)}{2}\int_{[0,T]}\int_{\mathbb{S}%
^{2}}\left\{ \left(\sum_{\ell }Z_{\ell }(x,t)\right)^{2}-\sum_{\ell }\frac{%
(2\ell +1)}{4\pi }C_{\ell }\right\} dxdt \\
&\quad -\sigma_1 \frac{\alpha _{20}\beta _{0}(u)}{2}\frac{1}{\sigma_1^2}\int_{[0,T]}\int_{%
\mathbb{S}^{2}}\left\{ \sum_{\ell }Z_{\ell }(x,t)\Delta \sum_{\ell ^{\prime }}Z_{\ell ^{\prime }}(x,t)-2\sigma_1^2\right\} dx dt\\
&=\sigma_1 \frac{\alpha _{00}\beta _{2}(u)}{2}\int_{[0,T]}\int_{\mathbb{S}%
^{2}}\left\{ \sum_{\ell} \sum_{\ell^{\prime}}Z_{\ell }(x,t)Z_{\ell^{\prime}
}(x,t)-\sum_{\ell }\frac{(2\ell +1)}{4\pi }C_{\ell }\right\} dxdt \\
&+\sigma_1 \frac{\alpha _{20}\beta _{0}(u)}{2}\frac{1}{\sigma_1^2}\int_{[0,T]}\int_{%
\mathbb{S}^{2}}\left\{ \sum_{\ell }\sum_{\ell ^{\prime }}Z_{\ell
}(x,t)\lambda _{\ell ^{\prime }}Z_{\ell ^{\prime }}(x,t)-\sum_{\ell } \frac{%
(2\ell +1)}{4\pi }C_{\ell }\lambda _{\ell }\right\} dxdt\\
&=\sigma_1 \frac{\alpha _{00}\beta _{2}(u)}{2}\int_{[0,T]}\left\{ \sum_{\ell }\int_{\mathbb{S}%
^{2}} Z_{\ell }(x,t)^{2} dx-4\pi\sum_{\ell }\frac{(2\ell +1)}{4\pi }C_{\ell }\right\} dt \\
&+\sigma_1 \frac{\alpha _{20}\beta _{0}(u)}{2}\frac{1}{\sigma_1^2}%
\int_{[0,T]}\left\{ \sum_{\ell }\lambda _{\ell }\int_{\mathbb{S}^{2}}Z_{\ell
}(x,t)^2 dx-4\pi \sum_{\ell } \lambda _{\ell } \frac{(2\ell +1)}{4\pi
}C_{\ell } \right\} dt\\
&=\sigma_1 \frac{\alpha _{00}\beta _{2}(u)}{2}\int_{[0,T]}\left\{ \sum_{\ell
}(2\ell+1)\widehat{C}_{\ell }(t)-\sum_{\ell }(2\ell +1)C_{\ell }\right\} dt
\\
&+\sigma_1 \frac{\alpha _{20}\beta _{0}(u)}{2}\frac{1}{\sigma_1^2}%
\int_{[0,T]}\left\{ \sum_{\ell }\lambda _{\ell }(2\ell+1)\widehat{C}_{\ell
}(t)- \sum_{\ell } \lambda _{\ell } (2\ell +1)C_{\ell } \right\} dt \\
&=\sigma_1 \frac{\alpha _{00}\beta _{2}(u)}{2}\int_{[0,T]}\sum_{\ell }(2\ell
+1)\left\{ \widehat{C}_{\ell }(t)-C_\ell\right\} dt \\
&+\sigma_1 \frac{\alpha _{20}\beta _{0}(u)}{2}\frac{1}{\sigma_1^2}%
\int_{[0,T]}\sum_{\ell }(2\ell +1)\ell (\ell +1)\left\{ \widehat{C}_{\ell
}(t)-C_\ell\right\} dt,
\end{flalign*}
where $\widehat{C}_{\ell }$ is as in (\ref{sampletris}). Also,
\begin{equation*}
\beta _{0}(u)=\phi (u)\text{ , }\beta _{2}(u)=\phi (u)(u^{2}-1)\text{ , }%
\alpha _{00}=\sqrt{\frac{\pi }{2}}\text{ , }\alpha _{02}=\frac{1}{2}\sqrt{%
\frac{\pi }{2}}
\end{equation*}%
whence%
\begin{equation*}
\frac{\alpha _{00}\beta _{2}(u)}{2}=\frac{1}{2}\sqrt{\frac{\pi }{2}}\phi
(u)(u^{2}-1)\text{ , }\frac{\alpha _{20}\beta _{0}(u)}{2}=\frac{1}{4}\sqrt{%
\frac{\pi }{2}}\phi (u)\text{ .}
\end{equation*}%
We can then write the second-order chaos more compactly as%
\begin{eqnarray*}
\mathcal{C}_T(u)[2] &=& \frac{\sigma _{1}}{2}\sqrt{\frac{\pi }{2}}\phi
(u)(u^{2}-1)\int_{[0,T]}\sum_{\ell }(2\ell +1)\left\{ \widehat{C}_{\ell
}(t)-C_{\ell }\right\} dt \\
&&+\frac{\sigma _{1}}{4}\sqrt{\frac{\pi }{2}}\phi (u)\frac{1}{\sigma _{1}^{2}%
}\int_{[0,T]}\sum_{\ell }(2\ell +1)\lambda _{\ell }\left\{ \widehat{C}_{\ell
}(t)-C_{\ell }\right\} dt \\
&=& \frac{\sigma _{1}}{2}\sqrt{\frac{\pi }{2}}\phi (u)\sum_{\ell }(2\ell
+1)\left\{ (u^{2}-1)+\frac{\lambda _{\ell }/2}{\sigma _{1}^{2}}\right\}
\int_{[0,T]} \left\{ \widehat{C}_{\ell }(t)-C_{\ell }\right\} dt
\end{eqnarray*}
thus concluding the proof. \endproof

\begin{remark}
\label{2chaos-hermite} \textrm{The second order chaos can be also written in
terms of Hermite polynomials, since
\begin{eqnarray*}
&& \int_{[0,T]}\left\{ \widehat{C}_{\ell }(t)-\mathbb{E}\widehat{C}_{\ell
}(t)\right\} dt \\
&&= \int_{[0,T]}\left\{\frac{1}{2\ell+1}\int_{\mathbb{S}^2}Z_\ell(x,t)^2dx-%
\frac{1}{2\ell+1}\int_{\mathbb{S}^2}\mathbb{E}[Z_\ell(x,t)^2]\right\} dt \\
&&= \int_{[0,T]}\left\{\frac{1}{2\ell+1}\int_{\mathbb{S}^2}Z_\ell(x,t)^2dx-%
\frac{1}{2\ell+1}\int_{\mathbb{S}^2}\frac{2\ell+1}{4\pi}C_{\ell
}(0)dx\right\} dt \\
&&= \int_{[0,T]}\frac{1}{2\ell+1}\int_{\mathbb{S}^2}\left\{Z_\ell(x,t)^2dx-%
\frac{2\ell+1}{4\pi}C_{\ell }(0)\right\} dxdt \\
&&= \int_{[0,T]}\frac{1}{2\ell+1}\frac{2\ell+1}{4\pi}C_{\ell }(0)\int_{%
\mathbb{S}^2}\left\{\widehat Z_\ell(x,t)^2dx-1\right\} dxdt \,,
\end{eqnarray*}
so that
\begin{flalign*}
&\mathcal{C}_T(u)[2]=\frac{\sigma _{1}}{2}\sqrt{\frac{\pi }{2}}\phi (u)\sum_{\ell }  \frac{C_{\ell }(0)(2\ell +1)}{4\pi } \left\{ (u^{2}-1)+\frac{\lambda _{\ell }/2}{\sigma _{1}^{2}}\right\} \int_0^T\int_{\mathbb{S}^{2}}H_{2}(\widehat Z_{\ell }(x,t))dxdt\text{,}
\end{flalign*}
as anticipated in (\ref{caos2}). }
\end{remark}

\begin{remark}[Non-asymptotic monochromatic field]
\label{2chaos_mono}\textrm{In the special case of monochromatic fields where
\begin{equation*}
C_{\ell }(0)\neq 0\Leftrightarrow \ell =\ell ^{\star },
\end{equation*}
we have that $\left ( \sigma^2_0= \frac{2\ell+1}{4\pi}C_\ell(0)=1\right )$ $%
\sigma^2 _{1}=\frac{\ell (\ell +1)}{2}$ and we get a straightforward
generalization of the standard non-asymptotic expression for the
second-order chaos for the boundary length of a time-dependent random
spherical harmonic, namely
\begin{eqnarray}
\mathcal{C}_T(u)[2] =\sqrt{\frac{\ell (\ell +1)}{2}}\frac{1}{2}\sqrt{\frac{%
\pi }{2}}u^{2}\phi (u)\int_{[0,T]}\int_{\mathbb{S}^{2}} H_2(Z_\ell(x,t))
dxdt.  \notag
\end{eqnarray}
(Note that $\widehat{Z}_\ell = Z_\ell$ in this case.) }
\end{remark}

\begin{remark}
\textrm{It is clear from (\ref{2chaos}) that the disappearance of the
second-order chaos at $u=0$ (closely related to the Berry's cancellation
phenomenon) does not occur for non-monochromatic space-time random fields --
although it does occur in the non-asymptotic monochromatic case (Remark \ref%
{2chaos_mono}). As we already showed in Section \ref{sec-MR}, the
cancellation can occur asymptotically (as $T\rightarrow \infty $) in some
cases of long range dependent fields where the memory parameter attains its
minimum on a single multipole $\ell^{\star}$; this can be viewed as a form
of asymptotic monochromatic behaviour. }
\end{remark}

From (\ref{caos2}) we have
\begin{flalign}\label{varlunga2}
&\operatorname{Var}\left( \mathcal{C}_{T}(u)[2]\right)  \nonumber\\
&=\frac{\sigma _{1}^{2}\pi }{8}\phi ^{2}(u)\mathbb{E}\left[ \left(
\sum_{\ell }(2\ell +1)\left\{ (u^{2}-1)+\frac{\lambda _{\ell }/2}{\sigma
_{1}^{2}}\right\} \frac{C_{\ell }(0)}{4\pi }\int_{[0,T]}\int_{\mathbb{S}^{2}}H_{2}(\widehat{Z}_{\ell }(x,t))dxdt\right) ^{2}\right] \nonumber \\
&=\frac{\sigma _{1}^{2}\pi }{8}\phi ^{2}(u)\sum_{\ell }(2\ell
+1)^{2}\left\{ (u^{2}-1)+\frac{\lambda _{\ell }/2}{\sigma _{1}^{2}}\right\}
^{2}\frac{C_{\ell }(0)^{2}}{(4\pi )^{2}}\mathbb{E}\left[ \left(
\int_{[0,T]}\int_{\mathbb{S}^{2}}H_{2}(\widehat{Z}_{\ell }(x,t))dxdt\right)
^{2}\right] \nonumber \\
&=\frac{\sigma _{1}^{2}\pi }{8}\phi ^{2}(u)\sum_{\ell }(2\ell
+1)^{2}\left\{ (u^{2}-1)+\frac{\lambda _{\ell }/2}{\sigma _{1}^{2}}\right\}
^{2}\frac{C_{\ell }(0)^{2}}{(4\pi )^{2}}\int_{[0,T]^{2}}\int_{\mathbb{S}^{2}\times \mathbb{S}^{2}}\mathbb{E}\left[ H_{2}(\widehat{Z}_{\ell
}(x,t))H_{2}(\widehat{Z}_{\ell }(y,s))\right] dxdtdyds \nonumber \\
&=\frac{\sigma _{1}^{2}\pi }{8}\phi ^{2}(u)\sum_{\ell }\left\{ (u^{2}-1)+\frac{\lambda _{\ell }/2}{\sigma _{1}^{2}}\right\}
^{2}\int_{[0,T]^{2}}\int_{\mathbb{S}^{2}\times \mathbb{S}^{2}}2\,\frac{(2\ell +1)^{2}}{(4\pi )^{2}}C_{\ell
}(t-s)^{2}P_{\ell }(\langle x,y\rangle
)^{2}dxdtdyds \nonumber \\
&=\frac{\sigma _{1}^{2}\pi }{8}\phi ^{2}(u)\sum_{\ell }\frac{2\,(2\ell
+1)^{2}}{(4\pi )^{2}}\left\{ (u^{2}-1)+\frac{\lambda _{\ell }/2}{\sigma
_{1}^{2}}\right\} ^{2}\int_{[0,T]^{2}}C_{\ell }(t-s)^{2}dtds\int_{\mathbb{S}^{2}\times \mathbb{S}^{2}}P_{\ell }(\langle x,y\rangle )^{2}dxdy \nonumber \\
&=\frac{\sigma _{1}^{2}\pi }{4}\phi ^{2}(u)\sum_{\ell =0}^{+\infty}(2\ell
+1)^{2}\left\{ (u^{2}-1)+\frac{\lambda _{\ell }/2}{\sigma
_{1}^{2}}\right\} ^{2}\int_{[0,T]^{2}}C_{\ell }(t-s)^{2}\,dt\,ds,
\end{flalign}
recalling that $\int_{\mathbb{S}^{2}\times \mathbb{S}^{2}}P_{\ell }(\langle
x,y\rangle )^{2}dxdy = (4\pi)^2/(2\ell+1)$. In view of (\ref{varlunga2}), we
will need the following result.

\begin{lemma}[Lemma 4.3 in \protect\cite{MRV:20}]
\label{lemma-var2nd} Fix $\ell \in \widetilde{ \mathbb{N}}$. If $2\beta
_{\ell}<1$, then%
\begin{equation*}
\lim_{T\rightarrow \infty }\frac{1}{T^{2-2\beta _{\ell }}}%
\int_{[0,T]^2}C_{\ell }^{2}(t-s)dtds=\frac{2C_{\ell }(0)^{2}}{%
(1-\beta_{\ell})(1-2\beta_{\ell})}\text{ .}
\end{equation*}
If $2\beta _{\ell}>1$, then
\begin{equation*}
\lim_{T\rightarrow \infty }\frac{1}{T}\int_{[0,T]^2}C_{\ell
}^{2}(t-s)dtds=\int_{\mathbb{R}}C_\ell(\tau)^2\, d\tau\,.
\end{equation*}
\end{lemma}

\begin{proposition}
\label{prop:var-2nd-chaos} For $2\beta_{\ell^*}<1$ and $\beta_0\le
\beta_{\ell^*}$ we have that
\begin{equation}
\lim_{T\rightarrow \infty }\frac{\mathop{\rm Var}\left( \mathcal{C}%
_{T}(u)[2]\right) }{T^{2-2\beta _{\ell^{\star}}}} =\frac{\sigma _{1}^{2}\pi
}{4}\phi ^{2}(u)\sum_{\ell \in \mathcal{I}^\star}\frac{(2\ell+1)^{2}C_%
\ell(0)^2}{(1-2\beta_\ell)(1-\beta_\ell)}\left\{ (u^{2}-1)+\frac{\lambda
_{\ell}/2}{\sigma_{1}^{2}}\right\}^2\,.
\end{equation}
For $2\beta_{\ell^*} >1$ and $2\beta_0 >1$ we have that
\begin{equation}  \label{T2}
\lim_{T\rightarrow \infty }\frac{\mathop{\rm Var}\left( \mathcal{C}%
_{T}(u)[2]\right) }{T} = \frac{\sigma _{1}^{2}\pi }{4}\phi
^{2}(u)\sum_{\ell=0}^{\infty}(2\ell+1)^{2} \left\{ (u^{2}-1)+\frac{\lambda
_{\ell}/2}{\sigma_{1}^{2}}\right\}^2 \int_{(-\infty, +\infty)}C_{\ell
}^{2}(\tau)d\tau.
\end{equation}
\end{proposition}

\proof
For $2\beta_{\ell^*}<1$ and $\beta_0\le \beta_{\ell^*}$, from Lemma \ref%
{lemma-var2nd}, bearing in mind (\ref{eqConv}), we can use Dominated
Convergence Theorem (as well as Lemmas 4.7 and 4.8 in \cite{MRV:20}) to get
\begin{flalign*}
&\lim_{T\to \infty}\frac{\operatorname{Var}\left(\mathcal{C}_T (u)[2]\right)}{T^{2-2\beta_{\ell^\star}}}\\
&= \lim_{T\to \infty} \frac{\sigma _{1}^{2}\pi }{4}\phi ^{2}(u)\sum_{\ell=0}^{\infty}\frac{T^{2-2\beta_{\ell}}(2\ell+1)^{2}}{T^{2-2\beta _{\ell^\star}}}\left\{ (u^{2}-1)+\frac{\lambda _{\ell}/2}{\sigma_{1}^{2}}\right\}^2 \int_{[0,T]^2}\frac{C_{\ell }^{2}(t-s)}{T^{2-2\beta _{\ell }}}dtds\\
&= \frac{\sigma _{1}^{2}\pi }{4}\phi ^{2}(u)\sum_{\ell=0}^{\infty}\lim_{T\to \infty} \frac{T^{2-2\beta_{\ell}}(2\ell+1)^{2}}{T^{2-2\beta _{\ell^\star}}}\left\{ (u^{2}-1)+\frac{\lambda _{\ell}/2}{\sigma_{1}^{2}}\right\}^2 \frac{C_{\ell }(0)^{2}}{(1-\beta_{\ell})(1-2\beta_{\ell})}dtds\\
&=  \frac{\sigma _{1}^{2}\pi }{4}\phi ^{2}(u)\sum_{\ell \in \mathcal{I}^\star}\frac{(2\ell+1)^{2}C_\ell(0)^2}{(1-2\beta_\ell)(1-\beta_\ell)}\left\{ (u^{2}-1)+\frac{\lambda _{\ell}/2}{\sigma_{1}^{2}}\right\}^2.
\end{flalign*}
The proof of (\ref{T2}) is analogous and hence omitted. \endproof

\subsubsection{Higher-order chaotic projections}

\label{high}

Let us investigate the asymptotic distribution, as $T\to +\infty$, of $%
\mathcal{C}_T(u)$ for $q\ge 3$. In the short memory case, it is trivial to
see that, as $T\to +\infty$,
\begin{equation*}
\text{Var}(\mathcal{C}_T(u)[q]) = O(T).
\end{equation*}
Note that the constants involved in the bound depend on $q$, but they are
uniformly square summable. The terms which are of smaller order are clearly
negligible; it is thus sufficient to establish that the fourth order
cumulants of the non-negligible chaotic components are $o(T^2)$. The proof
of this upper bound is standard and straightforward, following the same
steps as given for instance in \cite{MRV:20}.

The next Proposition refers to the long memory case and shows that all
chaotic components other than the leading one are uniformly negligible, in
the limit $T\to +\infty$. 

\begin{proposition}
\label{var-higher-chaoses} For $2\beta_{\ell^*} < \min\lbrace \beta_0,
1\rbrace$, as $T\to +\infty$,
\begin{equation*}
\sum_{q\geq 3}\mathop{\rm Var}(\mathcal{C}_{T}(u)[q])=O\left( T^{2-\frac52
\beta _{\ell ^{\star }}}\right).
\end{equation*}
\end{proposition}

\proof
We have
\begin{flalign*}
&\sum_{q\geq 3}\operatorname{Var}(\mathcal{C}_{T}(u)[q]) \\
=&\sigma _{1}^{2}\sum_{q\geq 3}\operatorname{Var}\left(
\sum_{m=0}^{q}\sum_{k=0}^{m}\frac{\alpha _{k,m-k}\beta _{q-m}(u)}{%
k!(m-k)!(q-m)!}	\right.\\
&\qquad\qquad\qquad\left.\int_{0}^{T}\int_{\mathbb{S}^{2}}H_{q-m}(Z(x,t))H_{k}(%
\widetilde{\partial }_{1,x}Z(x,t))H_{m-k}(\widetilde{\partial }%
_{2,x}Z(x,t))dxdt\right) \\
=&\sigma _{1}^{2}\sum_{q\geq
3}\sum_{m_{1}=0}^{q}\sum_{k_{1}=0}^{m_{1}}\sum_{m_{2}=0}^{q}%
\sum_{k_{2}=0}^{m_{2}}\frac{\alpha _{k_{1},m_{1}-k_{1}}\beta _{q-m_{1}}(u)}{%
k_{1}!(m_{1}-k_{1})!(q-m_{1})!}\frac{\alpha _{k_{2},m_{2}-k_{2}}\beta
_{q-m_{2}}(u)}{k_{2}!(m_{2}-k_{2})!(q-m_{2})!} \\
& \int_{\lbrack 0,T]^{2}}\int_{\mathbb{S}^{2}\times \mathbb{S}^{2}}%
\mathbb{E}\bigg[H_{q-m_{1}}(Z(x,t))H_{k_{1}}(\widetilde{\partial }%
_{1,x}Z(x,t))H_{m_{1}-k_{1}}(\widetilde{\partial }_{2,x}Z(x,t)) \\
&\qquad \qquad \qquad \qquad  H_{q-m_{2}}(Z(y,s))H_{k_{2}}(\widetilde{%
\partial }_{1,x}Z(y,s))H_{m_{2}-k_{2}}(\widetilde{\partial }_{2,x}Z(y,s))%
\bigg] dxdydtds.
\end{flalign*}
Hence we can write
\begin{flalign*}
\sum_{q\geq 3}\operatorname{Var}(\mathcal{C}_{T}(u)[q]) \leq &\sigma _{1}^{2}\sum_{q\geq
3}\sum_{i_{1}+i_{2}+i_{3}=q}\sum_{j_{1}+j_{2}+j_{3}=q}\frac{|\alpha
_{i_{1},i_{2}}\beta _{i_{3}}(u)|}{i_{1}!i_{2}!i_{3}!}\frac{|\alpha
_{j_{1},j_{2}}\beta _{j_{3}}(u)|}{j_{1}!j_{2}!j_{3}!}%
\,U_{q}(i_{1},i_{2},i_{3},j_{1},j_{2},j_{3})\,,
\end{flalign*}where $U_{q}(i_{1},i_{2},i_{3},j_{1},j_{2},j_{3})$ is a sum of
at most $q!$ terms of the type
\begin{equation}
\int_{\lbrack 0,T]^{2}}\int_{\mathbb{S}^{2}\times \mathbb{S}%
^{2}}\prod_{u=1}^{q}\mathbb{E}\left[ \widetilde{\partial }_{l_{u},x}Z(x,t)%
\widetilde{\partial }_{h_{u},x}Z(y,s)\right] \,dx\,dy\,dt\,ds\,,
\end{equation}%
where $l_{u},h_{u}\in \{0,1,2\}$ and by $\widetilde{\partial }%
_{l_{u},x}Z(x,t)$ we denote the normalized partial derivatives with respect
to the first or second variable (in our convention, if $l_u=0$ then $%
\widetilde{\partial }_{0,x}Z(x,t)=Z(x,t)$). In particular,
\begin{flalign*}
&|U_{q}(i_{1},i_{2},i_{3},j_{1},j_{2},j_{3})|\leq q!\,(4\pi
)^{2}\,\sum_{\ell _{1},\dots ,\ell _{q}=0}^{\infty }\frac{(2\ell _{1}+1)\ell
_{1}^{2}}{4\pi}\frac{(2\ell _{2}+1)\ell _{2}^{2}}{4\pi}\cdots \frac{(2\ell _{q}+1)\ell _{q}^{2}}{4\pi} \\
&\qquad\qquad\qquad\frac{C_{\ell_1}(0)C_{\ell_2}(0)\cdots C_{\ell_q}(0)\,T^{2-(\beta_{\ell_1}+\beta_{\ell_2}+\cdots+\beta_{%
\ell_q})}}{(1-\beta_{\ell_1}-\beta_{\ell_2}-\cdots-\beta_{\ell_q})(2-\beta_{%
\ell_1}-\beta_{\ell_2}-\cdots-\beta_{\ell_q})}+O(T)\\
&\leq \frac{q!\,(4\pi )^{2}\,T^{2-q\beta _{\ell ^{\star }}}}{%
\min_{q}\{(1-\beta _{\ell _{1}}-\beta _{\ell _{2}}-\cdots -\beta _{\ell
_{q}})(2-\beta _{\ell _{1}}-\beta _{\ell _{2}}-\cdots -\beta _{\ell _{q}})\}}%
\,\left( \sum_{\ell =0}^{\infty }\frac{C_\ell(0)(2\ell +1)\ell ^{2}}{4\pi}\right) ^{q}\,.
\end{flalign*}
Indeed, let us investigate the behavior of one of these terms:
\begin{flalign*}
& \int_{[0,T]^2}\int_{\mathbb{S}^2 \times \mathbb{S}^2}
\mathbb{E}\left[\partial_{2,x} Z(x,t)\partial_{2,y}
Z(y,s) \right]^2 \mathbb{E}\left[Z(x,t)\partial_{2,y} Z(y,s) %
\right]^{q-2}\,dx\,dy\,dt\,ds \\
& =\int_{[0,T]^2}\int_{\mathbb{S}^2 \times \mathbb{S}%
^2} \left(\sum_{\ell_1=0}^\infty C_{\ell_1}(t-s)\frac{(2\ell_1+1)}{4\pi}%
\partial_{2;x}\partial_{2;y}P_{\ell_1}(\langle x,y \rangle) \right)^2
\,\times \\
&\qquad\qquad \times\, \left(\sum_{\ell_3=0}^\infty C_{\ell_3}(t-s)\frac{%
(2\ell_3+1)}{4\pi}\partial_{2;y}P_{\ell_3}(\langle x,y \rangle) \right)^{q-2}
\\
& = \int_{[0,T]^2}\int_{\mathbb{S}^2 \times \mathbb{S}%
^2} \sum_{\ell_1,\ell_2=0}^\infty C_{\ell_1}(t-s)C_{\ell_2}(t-s)\frac{%
(2\ell_1+1)}{4\pi}\frac{(2\ell_2+1)}{4\pi}\\
&\qquad\qquad\qquad  \partial_{2;x}\partial_{2;y}P_{%
\ell_1}(\langle x,y \rangle)\partial_{2;x}\partial_{2;y}P_{\ell_2}(\langle
x,y \rangle) \\
&\qquad\qquad \sum_{\ell_3,\dots,\ell_q=0}^\infty C_{\ell_3}(t-s)\cdots
C_{\ell_q}(t-s)\frac{(2\ell_3+1)}{4\pi}\cdots\frac{(2\ell_q+1)}{4\pi}\\
&\qquad\qquad\qquad
\partial_{2;y}P_{\ell_3}(\langle x,y \rangle) \cdots
\partial_{2;y}P_{\ell_q}(\langle x,y \rangle) \\
& \leq(4\pi)^2\sigma _{1}^{2} \int_{[0,T]^2}
\sum_{\ell_1,\ell_2=0}^\infty C_{\ell_1}(t-s)C_{\ell_2}(t-s)\frac{(2\ell_1+1)%
}{4\pi}\frac{(2\ell_2+1)}{4\pi}\,\ell_1^2 \ell_2^2 \\
&\qquad\qquad \sum_{\ell_3,\dots,\ell_q=0}^\infty C_{\ell_3}(t-s)\cdots
C_{\ell_q}(t-s)\frac{(2\ell_3+1)}{4\pi}\cdots\frac{(2\ell_q+1)}{4\pi}\,
\ell_3 \cdots \ell_q \\
& = \sum_{\ell_1,\dots,\ell_q=0}^\infty\frac{%
(2\ell_1+1)}{4\pi}\frac{(2\ell_2+1)}{4\pi}\,\ell_1^2 \ell_2^2 \frac{%
(2\ell_3+1)}{4\pi}\cdots\frac{(2\ell_q+1)}{4\pi}\, \ell_3 \cdots \ell_q
\\
&\qquad \int_{[0,T]^2} C_{\ell_1}(t-s)\cdots C_{\ell_q}(t-s)\,dt\,ds \\
&=  \sum_{\ell_1,\dots,\ell_q=0}^\infty\frac{%
(2\ell_1+1)}{4\pi}\frac{(2\ell_2+1)}{4\pi}\, \frac{(2\ell_3+1)}{4\pi}\cdots%
\frac{(2\ell_q+1)}{4\pi} \ell_1^2 \ell_2^2 \, \ell_3 \cdots \ell_q
 \\
&\qquad\qquad\qquad\frac{C_{\ell_1}(0)C_{\ell_2}(0)\cdots C_{\ell_q}(0)\,T^{2-(\beta_{\ell_1}+\beta_{\ell_2}+\cdots+\beta_{%
\ell_q})}}{(1-\beta_{\ell_1}-\beta_{\ell_2}-\cdots-\beta_{\ell_q})(2-\beta_{%
\ell_1}-\beta_{\ell_2}-\cdots-\beta_{\ell_q})}+O(T)\,,
\end{flalign*}
where for the last equality we used \cite[Lemma 4.11]{MRV:20} and the fact
that, from (\ref{eqConv}), $\sum_{\ell=0}^{\infty }\frac{C_{\ell }(0)(2\ell
+1)\ell^2}{4\pi } <+\infty $. As a consequence,
\begin{eqnarray*}
&&\sum_{q\geq 3}\mathop{\rm Var}(\mathcal{C}_{T}(u)[q]) \\
&\leq &\sigma _{1}^{2}\sum_{q\geq
3}\sum_{i_{1}+i_{2}+i_{3}=q}\sum_{j_{1}+j_{2}+j_{3}=q}\frac{|\alpha
_{i_{1},i_{2}}\beta _{i_{3}}(u)|}{i_{1}!i_{2}!i_{3}!}\frac{|\alpha
_{j_{1},j_{2}}\beta _{j_{3}}(u)|}{j_{1}!j_{2}!j_{3}!} \\
&&\,\frac{q!\,(4\pi )^{2}\,T^{2-q\beta _{\ell ^{\star }}}}{%
\min_{q}\{(1-\beta _{\ell _{1}}-\beta _{\ell _{2}}-\cdots -\beta _{\ell
_{q}})(2-\beta _{\ell _{1}}-\beta _{\ell _{2}}-\cdots -\beta _{\ell _{q}})\}}
\\
&& \qquad\qquad\qquad\left( \sum_{\ell =0}^{\infty }\frac{C_\ell(0)(2\ell
+1)\ell ^{2}}{4\pi}\right) ^{q} \\
&=&\sigma _{1}^{2}\sum_{q\geq 3}\frac{q!\,(4\pi )^{2}\,T^{2-q\beta _{\ell
^{\star }}}}{\min_{q}\{(1-\beta _{\ell _{1}}-\beta _{\ell _{2}}-\cdots
-\beta _{\ell _{q}})(2-\beta _{\ell _{1}}-\beta _{\ell _{2}}-\cdots -\beta
_{\ell _{q}})\}} \\
&&\qquad \left( \sum_{i_{1}+i_{2}+i_{3}=q}\frac{|\alpha _{i_{1},i_{2}}\beta
_{i_{3}}(u)|}{i_{1}!i_{2}!i_{3}!}\left( \sum_{\ell =0}^{\infty }\frac{%
C_\ell(0)(2\ell +1)\ell ^{2}}{4\pi}\right) ^{\frac{i_{1}+i_{2}+i_{3}}{2}%
}\right) ^{2} \\
&=&\frac{\sigma _{1}^{2}\,(4\pi )^{2}\,T^{2-\frac52\beta _{\ell ^{\star }}}}{%
\min_{q}\{(1-\beta _{\ell _{1}}-\beta _{\ell _{2}}-\cdots -\beta _{\ell
_{q}})(2-\beta _{\ell _{1}}-\beta _{\ell _{2}}-\cdots -\beta _{\ell _{q}})\}}
\\
&&\qquad \sum_{q\geq 3}q!\,\left( \sum_{i_{1}+i_{2}+i_{3}=q}\frac{|\alpha
_{i_{1},i_{2}}\beta _{i_{3}}(u)|}{i_{1}!i_{2}!i_{3}!}\left( \sum_{\ell
=0}^{\infty }\frac{C_\ell(0)(2\ell +1)\ell ^{2}}{4\pi \,
T^{\left(1-\frac5{2q}\right)\beta_{\ell^{\star}}}}\right)^{ \frac{%
i_{1}+i_{2}+i_{3}}{2}}\right) ^{2}\,.
\end{eqnarray*}%
So that for each $\varepsilon>0$ there exists $T_\varepsilon>0$ such that
\begin{eqnarray*}
\left( \sum_{\ell =0}^{\infty }\frac{C_\ell(0)(2\ell +1)\ell ^{2}}{4\pi \,
T^{\left(1-\frac5{2q}\right)\beta_{\ell^{\star}}}}\right) ^{\frac{%
i_{1}+i_{2}+i_{3}}{2}} &=&\left( \sum_{\ell =0}^{\infty }\frac{%
C_\ell(0)\,(2\ell +1)\ell ^{2}}{4\pi \,
T^{\left(1-\frac5{2q}\right)\beta_{\ell^{\star}}}}\right) ^{q/2}<
\varepsilon^{q/2} \,,
\end{eqnarray*}%
for each $T \ge T_\varepsilon$. Hence
\begin{equation*}
\sum_{q\geq 3}q!\left( \sum_{i_{1}+i_{2}+i_{3}=q}\frac{|\alpha
_{i_{1},i_{2}}\beta _{i_{3}}(u)|}{i_{1}!i_{2}!i_{3}!}\left( \sum_{\ell
=0}^{\infty }\frac{C_\ell(0)\,(2\ell +1)\ell ^{2}}{4\pi \,
T^{\left(1-\frac5{2q}\right)\beta_{\ell^{\star}}}}\right) ^{\frac{%
i_{1}+i_{2}+i_{3}}{2}}\right) ^{2}
\end{equation*}%
\begin{flalign*}
&=\sum_{q\geq 3}q!\sum_{i_{1}+i_{2}+i_{3}=q}\sum_{j_{1}+j_{2}+j_{3}=q}\frac{%
|\alpha _{i_{1},i_{2}}\beta _{i_{3}}(u)|}{i_{1}!i_{2}!i_{3}!}\frac{|\alpha
_{j_{1},j_{2}}\beta _{j_{3}}(u)|}{j_{1}!j_{2}!j_{3}!}\\
&\qquad\qquad\qquad\times\left( \sum_{\ell
=0}^{\infty }\frac{C_\ell(0)\,(2\ell +1)\ell ^{2}}{4\pi \, T^{\left(1-\frac5{2q}\right)\beta_{\ell^{\star}}}}\right) ^{q}
\\
&\leq \sum_{q\geq 3}q!\sum_{i_{1}+i_{2}+i_{3}=q}\sum_{j_{1}+j_{2}+j_{3}=q}%
\frac{|\alpha _{i_{1},i_{2}}\beta _{i_{3}}(u)|}{i_{1}!i_{2}!i_{3}!}\frac{%
|\alpha _{j_{1},j_{2}}\beta _{j_{3}}(u)|}{j_{1}!j_{2}!j_{3}!}\varepsilon
^{q}.
\end{flalign*}Now, arguing as in \cite[Section 6.2.2]{DNPR:16}, we have that
the previous quantity is equal to%
\begin{eqnarray*}
&&\sum_{q\geq 3}q!\sum_{i_{1}+i_{2}+i_{3}=q}\sum_{j_{1}+j_{2}+j_{3}=q}\frac{%
|\alpha _{i_{1},i_{2}}\beta _{i_{3}}(u)|}{i_{1}!i_{2}!i_{3}!}\frac{|\alpha
_{j_{1},j_{2}}\beta _{j_{3}}(u)|}{j_{1}!j_{2}!j_{3}!}\varepsilon ^{q} \\
&=&\sum_{q\geq 3}\varepsilon
^{q}\sum_{i_{1}+i_{2}+i_{3}=q}\sum_{j_{1}+j_{2}+j_{3}=q}\sqrt{%
(i_{1}+i_{2}+i_{3})!}\sqrt{(j_{1}+j_{2}+j_{3})!} \\
&&\qquad\qquad\qquad\times \frac{|\alpha _{i_{1},i_{2}}\beta _{i_{3}}(u)|}{%
i_{1}!i_{2}!i_{3}!}\frac{|\alpha _{j_{1},j_{2}}\beta _{j_{3}}(u)|}{%
j_{1}!j_{2}!j_{3}!} \\
&=&\sum_{q\geq 3}\varepsilon
^{q}\sum_{i_{1}+i_{2}+i_{3}=q}\sum_{j_{1}+j_{2}+j_{3}=q}\sqrt{\frac{%
(i_{1}+i_{2}+i_{3})!}{i_{1}!i_{2}!i_{3}!}}\sqrt{\frac{(j_{1}+j_{2}+j_{3})!}{%
j_{1}!j_{2}!j_{3}!}} \\
&&\qquad\qquad\qquad\times\frac{|\alpha _{i_{1},i_{2}}\beta _{i_{3}}(u)|}{%
\sqrt{i_{1}!i_{2}!i_{3}!}}\frac{|\alpha _{j_{1},j_{2}}\beta _{j_{3}}(u)|}{%
\sqrt{j_{1}!j_{2}!j_{3}!}} \\
&\leq &\sum_{q\geq 3}\varepsilon ^{q}\sqrt{\sum_{i_{1}+i_{2}+i_{3}=q}%
\sum_{j_{1}+j_{2}+j_{3}=q}\frac{(i_{1}+i_{2}+i_{3})!}{i_{1}!i_{2}!i_{3}!}%
\frac{(j_{1}+j_{2}+j_{3})!}{j_{1}!j_{2}!j_{3}!}} \\
&&\quad \times \sqrt{\sum_{i_{1}+i_{2}+i_{3}=q}\sum_{j_{1}+j_{2}+j_{3}=q}%
\frac{|\alpha _{i_{1},i_{2}}\beta _{i_{3}}(u)|^{2}}{i_{1}!i_{2}!i_{3}!}\frac{%
|\alpha _{j_{1},j_{2}}\beta _{j_{3}}(u)|^{2}}{j_{1}!j_{2}!j_{3}!}} \\
&\leq &\sum_{q\geq 3}K_{1}(q)\varepsilon ^{q}\sqrt{%
\sum_{i_{1}+i_{2}+i_{3}=q}\sum_{j_{1}+j_{2}+j_{3}=q}\frac{%
(i_{1}+i_{2}+i_{3})!}{i_{1}!i_{2}!i_{3}!}\frac{(j_{1}+j_{2}+j_{3})!}{%
j_{1}!j_{2}!j_{3}!}}
\end{eqnarray*}%
where%
\begin{equation*}
K_{1}(q)=\sum_{i_{1}+i_{2}+i_{3}=q}\frac{|\alpha _{i_{1},i_{2}}\beta
_{i_{3}}(u)|^{2}}{i_{1}!i_{2}!i_{3}!}\leq const,\text{ }
\end{equation*}%
uniformly in $q$ note that this is the variance of the $q-$th order chaos
for the expansion of the boundary length for a unit variance spherical
random field. On the other hand,
\begin{eqnarray*}
&&\sqrt{\sum_{i_{1}+i_{2}+i_{3}=q}\sum_{j_{1}+j_{2}+j_{3}=q}\frac{%
(i_{1}+i_{2}+i_{3})!}{i_{1}!i_{2}!i_{3}!}\frac{(j_{1}+j_{2}+j_{3})!}{%
j_{1}!j_{2}!j_{3}!}} \\
&&=\sum_{i_{1}+i_{2}+i_{3}=q}\frac{(i_{1}+i_{2}+i_{3})!}{i_{1}!i_{2}!i_{3}!}%
=3^{q},
\end{eqnarray*}
whence%
\begin{eqnarray*}
&&\sum_{q\geq 3}K_{1}(q)\varepsilon ^{q}\sqrt{\sum_{i_{1}+i_{2}+i_{3}=q}%
\sum_{j_{1}+j_{2}+j_{3}=q}\frac{(i_{1}+i_{2}+i_{3})!}{i_{1}!i_{2}!i_{3}!}%
\frac{(j_{1}+j_{2}+j_{3})!}{j_{1}!j_{2}!j_{3}!}} \\
&=&\sum_{q\geq 3}K_{1}(q)\varepsilon ^{q}3^{q}<\infty \text{ ,}
\end{eqnarray*}%
since one can choose $\varepsilon <\frac{1}{3}$. Consequently, we just
proved that
\begin{equation}
\sum_{q\geq 3}\mathop{\rm Var}(\mathcal{C}_{T}(u)[q])=O\left( T^{2-\frac52
\beta _{\ell ^{\star }}}\right)=o\left(T^{2-2 \beta _{\ell ^{\star }}}\right)
\label{var-high}
\end{equation}
thus concluding the proof. \endproof

\subsection{Proof of Theorem \protect\ref{main-theorem}}

We will need the following well known result.

\begin{theorem}[\protect\cite{DM79, Ta:79}]
\label{DMtheorem} Let $\xi(t)$, $t\in\mathbb{R}$, be a real measurable
mean-square continuous stationary Gaussian process with mean $\mathbb{E}%
\left[\xi(t)\right]$ and covariance function $\rho(t-s)=\rho(|t-s|)=%
\mathop{\rm Cov}(\xi(t),\xi(s))$. Moreover, assume that
\begin{equation}  \label{eq:DM79}
\rho(t-s)=\frac{L(|t-s|)}{|t-s|^{\beta}}\,, \qquad \text{with} \quad
0<\beta<1,
\end{equation}
where $L$ is a slowly varying function. Let $F: \mathbb{R} \rightarrow
\mathbb{R}$ be a Borel function such that $\mathbb{E}\left[F(N)^2\right]%
<+\infty$, where $N$ is a standard Gaussian random variable. Then it is a
well known fact that can be expanded as follows
\begin{equation*}
F(\xi)=\sum_{k=0}^\infty \frac{b_k}{k!}H_k(\xi)\,, \quad \text{where} \quad
b_k=\int_{\mathbb{R}}F(\xi)H_k(\xi)\phi(\xi)d\xi\,.
\end{equation*}
Assume there exists an integer $r$, the so-called Hermitian rank, such that $%
b_0=b_1=\cdots=b_{r-1}=0$ and $b_r\ne0$. Then , if $\beta\in(0,1/r)$, we
have that the finite-dimensional distributions of the random process
\begin{equation*}
X_T(s)=\frac{1}{T^{1-\beta r /2}L(T)^{r/2}}\int_{0}^{Ts}\,\left[F(\xi(t))-b_0%
\right] \, dt \,, \qquad 0\le s \le1\,,
\end{equation*}
converge weakly, as $T\rightarrow \infty$, to the ones of the Rosenblatt
process of order $r$, that is
\begin{equation*}
X_{\beta}(s):=\frac{b_r}{r!}\int_{(\mathbb{R}^r)^{\prime }}\,\frac{%
e^{i(\lambda_1+\cdots+\lambda_r)s}-1}{i(\lambda_1+\cdots+\lambda_r)}\frac{%
W(d\lambda_1)\cdots W(d\lambda_r)}{|\lambda_1\cdots\lambda_r|^{(1-\beta)/2}}
\, dt \,, \qquad 0\le s \le1\,,
\end{equation*}
where $W$ is a complex Gaussian white noise.
\end{theorem}

\begin{proof}[Proof of Theorem \ref{main-theorem}]
Recall that $2\beta_{\ell^\star}<\min(\beta_0,1)$.  From Lemma \ref{lemvar1} we have
$$
\lim_{T\to\infty} \frac{\operatorname{Var}(\mathcal C_T(u)[1])}{T^{2-2\beta_{\ell^\star}}} = 0\,.
$$
Moreover, thanks to Proposition \ref{var-higher-chaoses},
$$
\lim_{T\to\infty} \frac{\sum_{q\ge 3} \operatorname{Var}(\mathcal C_T(u)[q])}{T^{2-2\beta_{\ell^\star}}} = 0\,,
$$
so that, recalling also Proposition \ref{prop:var-2nd-chaos},
\begin{equation}\label{2chaosA}
\frac{\mathcal{C}_T(u)}{T^{1-\beta_{\ell^\star}}}=\frac{\mathcal{C}_T(u)[2]}{T^{1-\beta_{\ell^\star}}}+o_{\mathbb P}(1)\,.
\end{equation}
Moreover, since in $L^2(\Omega)$ we have the following equality (recall Remark \ref{2chaos-hermite})
\begin{flalign*}
&\mathcal{C}_T(u)[2]=\frac{\sigma _{1}}{2}\sqrt{\frac{\pi }{2}}%
\phi (u)\sum_{\ell }  \frac{C_{\ell }(0)(2\ell +1)}{4\pi } \left\{ (u^{2}-1)+\frac{\lambda _{\ell }/2}{%
\sigma _{1}^{2}}\right\} \int_0^T\int_{\mathbb{S}^{2}}H_{2}(\widehat Z_{\ell }(x,t))dxdt\text{,}
\end{flalign*}
it holds that
\begin{flalign}  \label{eq:var2spectral}
\frac{\mathcal{C}_T(u)[2]}{T^{1-\beta_{\ell^\star}}}&=
\sum_{\ell \in \mathcal{I}^\star} \frac{\sigma_1}{2}\sqrt{\frac\pi2} \phi(u) C_\ell(0) \left\{ (u^{2}-1)+\frac{\lambda _{\ell}/2}{\sigma_{1}^{2}}\right\}\notag\\
&\qquad \times \sum_{m=-\ell}^{\ell} \frac{1}{T^{1-\beta_{\ell^\star}}}\int_{0}^T H_2\left(\widehat a_{\ell m}(t)\right)dt
+o_{\mathbb P}(1)\,,
\end{flalign}
where $\hat a_{\ell m}(t):= a_{l m}(t) / \sqrt{C_\ell(0)}$.
Indeed, recalling Proposition \ref{prop:var-2nd-chaos}, we have that
\begin{equation}\label{th21}
\lim_{T\rightarrow \infty }\frac{\operatorname{Var}\left( \mathcal{C}_{T}(u)[2]\right) }{T^{2-2\beta _{\ell^{\star}}}}
=\frac{%
\sigma _{1}^{2}\pi }{4}\phi ^{2}(u)\sum_{\ell \in \mathcal{I}^\star}\frac{%
(2\ell+1)^{2}C_\ell(0)^2}{(1-2\beta_\ell)(1-\beta_\ell)}%
\left\{ (u^{2}-1)+\frac{\lambda _{\ell}/2}{\sigma_{1}^{2}}%
\right\}^2\,.
\end{equation}

and hence that
\begin{flalign*}
&\lim_{T\to\infty}\mathbb E\left [\left(\frac{\mathcal{C}_T(u)[2]}{T^{1-\beta_{\ell^\star}}} -\frac{1}{T^{1-\beta_{\ell^\star}}}\sum_{\ell\in \mathcal{I}^\star}\sum_{m=-\ell}^\ell \frac{J_2(u)}{2} C_\ell(0)\int_0^T H_2(\hat a_{\ell, m}(t))\, dt \right)^2\right ]\\
&=\lim_{T\to\infty}\frac{\sigma _{1}^{2}\pi }{4\,T^{2-2\beta_{\ell^\star}}}\phi ^{2}(u)\sum_{\ell \notin \mathcal{I}^\star}(2\ell
+1)^{2}\left\{ (u^{2}-1)+\frac{\lambda _{\ell }/2}{\sigma
_{1}^{2}}\right\} ^{2}\int_{[0,T]^{2}}C_{\ell }(t-s)^{2}dtds\,.
\end{flalign*}
From \eqref{2chaosA} and \eqref{eq:var2spectral}, in order to understand the asymptotic distribution of $\mathcal C_T(u)$, it suffices to investigate the leading term on the right hand side of \eqref{eq:var2spectral}.
Recall Assumption \ref{basic3}, for $\ell\in \mathcal I^\star$ we have that
$$
C_{\ell}(\tau)= \frac{G_\ell(\tau)}{(1+|\tau|)^{\beta_{\ell^\star}}} \,,
$$
where in particular $G_\ell$ is a slowly varying function.
Hence, setting $\xi(t)=a_{\ell, m}(t)$ in Theorem \ref{DMtheorem}, we automatically have that $\rho=\rho_\ell=C_{\ell}$, $L=L_\ell=G_\ell$ and, as a consequence, that
\begin{equation*}
X_T^{\ell, m}:=\frac{1}{T^{1-\beta_{\ell^\star}}}\int_0^T H_2(\widehat a_{\ell, m}(t))\, dt
\overset{d}{\longrightarrow} \frac{X_{m;\beta_{\ell^\star}}}{a(\beta_{\ell^\star})}\,, \qquad \text{as } T \rightarrow
\infty\,,
\end{equation*}
for all $m=-\ell,\dots,\ell$, where, for each $m$,
$X_{m;\beta_{\ell^\star}}$
is a standard Rosenblatt random variable \eqref{Xbeta} of parameter $\beta_{\ell^\star}$.
Moreover, since the $X_T^{\ell, m}$ are all independent for each $T$ we have that
\begin{flalign*}
&\widetilde{\mathcal{C}}_{T}(u)=\sqrt{\frac{T^{2-2\beta_{\ell^\star}}}{\operatorname{Var} \mathcal{C}_{T}(u)[2]}} \sum_{\ell \in \mathcal{I}^\star} \frac{\sigma_1}{2}\sqrt{\frac\pi2} \phi(u) C_\ell(0) \left\{ (u^{2}-1)+\frac{\lambda _{\ell}/2}{\sigma_{1}^{2}}\right\}\\
&\qquad\qquad\qquad \times \sum_{m=-\ell}^{\ell} \frac{\int_{0}^T H_2\left(\widehat a_{lm}(t)\right)dt}{T^{1-\beta_{\ell^\star}}}+o_{\mathbb{P}}(1)
\end{flalign*}
$$
\mathop{\longrightarrow}^d \left(\frac{\sigma_{1}^{2}\pi }{4}\phi ^{2}(u)\sum_{\ell \in \mathcal{I}^\star}  \frac{(2\ell+1)^{2} C_\ell(0)^2}{(1-2\beta_\ell)(1-\beta_\ell)} \left\{ (u^{2}-1)+\frac{\lambda _{\ell}/2}{\sigma_{1}^{2}}\right\}^2 \right)^{-1/2}
$$
$$
\times\sum_{\ell \in \mathcal{I}^\star} \frac{\sigma_1^2}{2}\sqrt{\frac\pi2} \phi(u) C_\ell(0) \left\{ (u^{2}-1)+\frac{\lambda _{\ell}/2}{\sigma_{1}^{2}}\right\}\sum_{m=-\ell}^{\ell} \frac{X_{m,\beta_{\ell^\star}}}{a(\beta_{\ell^\star})}
$$
$$
\mathop{=}^{d}\sum_{\ell \in \mathcal{I}^\star} \frac{C_\ell(0)}{\sqrt{v^\star}}\left\{ (u^{2}-1)+\frac{\lambda _{\ell}/2}{\sigma_{1}^{2}}\right\}V_{2\ell+1}(1,\dots,1;\beta_{\ell^\star})\,,
$$
where
$$
v^\star=a(\beta_{\ell^\star})^2  \sum_{\ell \in \mathcal{I}^\star}  \frac{2(2\ell+1)^{2} C_\ell(0)^2}{(1-2\beta_\ell)(1-\beta_\ell)} \left\{ (u^{2}-1)+\frac{\lambda _{\ell}/2}{\sigma_{1}^{2}}\right\}^2
$$
and the proof is concluded.
\end{proof}

Some auxiliary results are collected in the four appendixes that follow.

\appendix

\section{Covariance structure}

\label{app-cov}

In this Section we collect technical results on the covariance structure of
the field $(Z, \nabla Z)$.

\begin{lemma}\label{lemcov}
Let $Z$ be a space-time spherical random field satisfying Assumption \ref{basic} and Assumption \ref{regular}. Then for all points $x=(\theta _{x},\varphi _{x}),$ $y=(\theta _{y},\varphi
_{y})\in \mathbb{S}^{2}\setminus \{N,S\}$, the covariance structure of
$(Z, \nabla Z)$ is
\begin{flalign*}
\mathbb{E}\left[ Z(x,t)Z(y,s)\right] = \Gamma(\langle x, y\rangle, t-s),
\end{flalign*}

\begin{equation*}
\mathbb{E}\left[ Z(x,t)\partial _{1;y}Z(y,s)\right]
\end{equation*}%
\begin{equation*}
=\sum_{\ell =0}^{\infty }\frac{2\ell +1}{4\pi }C_{\ell }(t-s)P_{\ell
}^{\prime }(\left\langle x,y\right\rangle )\left\{ -\cos \theta _{x}\sin
\theta _{y}+\sin \theta _{x}\cos \theta _{y}\cos (\varphi _{x}-\varphi
_{y})\right\} \text{ ,}
\end{equation*}

\begin{equation*}
\mathbb{E}\left[ Z(x,t)\partial _{2;y}Z(y,s)\right]
\end{equation*}%
\begin{equation*}
=\sum_{\ell =0}^{\infty }\frac{2\ell +1}{4\pi }C_{\ell }(t-s)P_{\ell
}^{\prime }(\left\langle x,y\right\rangle )\left\{ \sin \theta _{x}\sin
(\varphi _{x}-\varphi _{y})\right\} \text{ ,}
\end{equation*}%
and moreover
\begin{equation*}
\mathbb{E}\left[ \partial _{1;x}Z(x,t)\partial _{1;y}Z(y,s)\right] =
\end{equation*}%
\begin{eqnarray*}
&&\sum_{\ell =0}^{\infty }\frac{2\ell +1}{4\pi }C_{\ell }(t-s)P_{\ell
}^{\prime \prime }(\left\langle x,y\right\rangle )\left\{ -\cos \theta
_{x}\sin \theta _{y}+\sin \theta _{x}\cos \theta _{y}\cos (\varphi
_{x}-\varphi _{y})\right\}  \\
&&\times \left\{ -\sin \theta _{x}\cos \theta _{y}+\cos \theta _{x}\sin
\theta _{y}\cos (\varphi _{x}-\varphi _{y})\right\}  \\
&&+\sum_{\ell =0}^{\infty }\frac{2\ell +1}{4\pi }C_{\ell }(t-s)P_{\ell
}^{\prime }(\left\langle x,y\right\rangle )\left\{ \sin \theta _{x}\sin
\theta _{y}+\cos \theta _{x}\cos \theta _{y}\cos (\varphi _{x}-\varphi
_{y})\right\} \text{ ,}
\end{eqnarray*}%
\begin{equation*}
\mathbb{E}\left[ \partial _{1;x}Z(x,t)\partial _{2;y}Z(y,s)\right] =
\end{equation*}%
\begin{eqnarray*}
&&\sum_{\ell =0}^{\infty }\frac{2\ell +1}{4\pi }C_{\ell }(t-s)P_{\ell
}^{\prime \prime }(\left\langle x,y\right\rangle ) \\
&&\times \left\{ -\sin \theta _{x}\cos \theta _{y}+\cos \theta _{x}\sin
\theta _{y}\cos (\varphi _{x}-\varphi _{y})\right\} \left\{ \sin \theta
_{x}\sin \theta _{y}\sin (\varphi _{x}-\varphi _{y})\right\}  \\
&&+\sum_{\ell =0}^{\infty }\frac{2\ell +1}{4\pi }C_{\ell }(t-s)P_{\ell
}^{\prime }(\left\langle x,y\right\rangle )\cos \theta _{x}\sin \theta
_{y}\sin (\varphi _{x}-\varphi _{y})\text{ },
\end{eqnarray*}%
\begin{equation*}
\mathbb{E}\left[ \partial _{2;x}Z(x,t)\partial _{2;y}Z(y,s)\right] =
\end{equation*}%
\begin{eqnarray*}
&&-\sum_{\ell =0}^{\infty }\frac{2\ell +1}{4\pi }C_{\ell }(t-s)P_{\ell
}^{\prime \prime }(\left\langle x,y\right\rangle )\sin \theta _{x}\sin
\theta _{y}\sin ^{2}(\varphi _{x}-\varphi _{y}) \\
&&+\sum_{\ell =0}^{\infty }\frac{2\ell +1}{4\pi }C_{\ell }(t-s)P_{\ell
}^{\prime }(\left\langle x,y\right\rangle )\sin \theta _{x}\sin \theta
_{y}\cos (\varphi _{x}-\varphi _{y})\text{ }.
\end{eqnarray*}
\end{lemma}

The proof of these results is entirely analogous to the one given in the Appendix of \cite{CM20} and hence omitted; note that in the latter reference the definition of $\partial _{2;y}$ differs by a factor $\frac{1}{\sin \theta_{y}}$, i.e., covariant derivatives are used in the computations.

\section{Measurability issues}

We need to add (at least) the following assumption, which is equivalent to
assume that, for every $t\in \mathbb{R}$, the random field $Z(\cdot, t):%
\mathbb{S}^2\to \mathbb{R}$ is a.s. $\mathcal{C}^1$. Note that the $a.s.$
depends on $t$.

\begin{condition}
\label{assreg} For every $t\in \mathbb{R}$, the function $\mathbb{S}^2\times
\mathbb{S}^2 \ni (x,y) \mapsto \Gamma(\cos d(x,y), t)$ is $\mathcal{C}^2$.
\end{condition}

From now on we assume that $Z$ satisfies Assumption \ref{assreg} and
Assumption \ref{basic}. Let $u\in \mathbb{R}$ be a fixed threshold, for $%
t\in \mathbb{R}$ we consider the level set $Z(\cdot, t)^{-1}(u):=\lbrace
x\in \mathbb{S}^2 : Z(x,t) = u\rbrace$ which is an a.s. $\mathcal{C}^1$
manifold of dimension $1$ ($a.s.$ depends on $t$). Indeed, for every $x\in
Z(\cdot, t)^{-1}(u)$ the covariance of $(Z(x,t), \nabla_x Z(x,t))$ is
non-degenerate, hence Bulinskaya's lemma ensures that there exists $%
\Omega_t\subseteq \Omega$, $\mathbb{P}(\Omega_t)=1$, such that for every $%
\omega\in \Omega_t$, the value $u$ is regular for $Z(\cdot, t)(\omega)$,
i.e.
\begin{equation*}
\nabla_x Z(x, t)(\omega)\ne 0 \text{ for every } x \text{ such that }
Z(x,t)(\omega)=u.
\end{equation*}
Hence, on $\Omega_t$ we can define
\begin{equation*}
\mathcal{L}_u(t):=\text{length}(Z(\cdot, t)^{-1}(u)).
\end{equation*}
By stationarity, the law of $\mathcal{L}_u(t)$ does not depend on $t$, in
particular $\mathbb{E}[\mathcal{L}_u(t)]$ does not depend on $t$, and can be
computed via the Kac-Rice formula \cite[Theorem 6.8]{AW:09} or the Gaussian
Kinematic Formula \cite[Theorem 13.2.1]{adlertaylor} to be
\begin{equation}  \label{meanGKF}
\mathbb{E}[\mathcal{L}_u(t)] = \sigma_1 \cdot 2\pi e^{-u^2/2},
\end{equation}
where $\sigma_1$ is defined as in (\ref{sigma1}). In order to define our
functional of interest and prove that it is a random variable, we need the
following technical result.

\begin{lemma}
\label{lemmeas} For every $T>0$, the set
\begin{eqnarray*}
A^T &:=& \left \lbrace (\omega, t)\in \Omega\times [0,T]: \lbrace x\in
\mathbb{S}^2 : |Z(x,t)(\omega) - u | =0, |\nabla Z(x,t)(\omega) |=0\rbrace
=\emptyset \right \rbrace \\
&=&\lbrace (\omega, t)\in \Omega \times [0,T] : \text{ the value } u \text{
is regular for } Z(\cdot, t)(\omega)\rbrace.
\end{eqnarray*}
is measurable, i.e., $A^T\in \mathcal{F}\otimes \mathcal{B}([0,T])$.
\end{lemma}

\begin{proof}
Fix $\lbrace x_i\rbrace_{i\in \mathbb N}$ a dense sequence in $\mathbb S^2$; for $n,k\in \mathbb N$, define the set
\begin{equation}
A^T_{n,k}:= \bigcap_{i\in \mathbb N} \left \lbrace (\omega, t)\in \Omega\times [0,T]: |Z(x_i, t, \omega) - u | \ge \frac{1}{n} \right \rbrace \cup \left \lbrace  (\omega, t)\in \Omega\times [0,T] : |\nabla Z(x_i, t, \omega)| \ge \frac{1}{k} \right \rbrace
\end{equation}
which is measurable by construction, i.e. $A^T_{n,k}\in \mathcal F\otimes \mathcal B([0,T])$. We will show that $A^T$ equals 
\begin{equation}
 \bigcup_{k,n\in \mathbb N} A^T_{n,k} =: \widehat A^T,
\end{equation}
so that, in particular, $A^T$ is measurable. Indeed,
\begin{equation}
A^T\subseteq \widehat A^T,
\end{equation}
because if $(\omega,t)\in A^T$, then
\begin{equation}
\inf_{x\in \mathbb S^2: Z(x,t)(\omega)=u} |\nabla Z(x,t)(\omega)| =: l >0
\end{equation}
hence by continuity of $x\mapsto Z(x,t)(\omega)$ and $x\mapsto \nabla Z(x,t)(\omega)$ there exist $\tilde k, \tilde n\in \mathbb N$ s.t.
\begin{equation}
\inf_{x\in \mathbb S^2: |Z(x,t)(\omega) - u| < \frac{1}{\tilde n}} |\nabla Z(x,t)(\omega)| > l/2 > \frac{1}{\tilde k}.
\end{equation}
Thus $(\omega,t)\in A^T_{\tilde n, \tilde k}$. On the other hand, if $(\omega,t)\in \widehat A^T$, then there exist $\tilde n, \tilde k\in \mathbb N$ s.t. $(\omega,t)\in A^T_{\tilde n, \tilde k}$, that is,
\begin{equation}
\inf_{i\in \mathbb N : |Z(x_i, t)(\omega) - u|< 1/\tilde n} |\nabla Z(x_i, t)(\omega) | \ge 1/\tilde k.
\end{equation}
By continuity of $Z$ and $\nabla Z$
\begin{equation}
\inf_{x\in \mathbb S^2 : |Z(x, t)(\omega) - u|< 1/2\tilde n} |\nabla Z(x, t)(\omega) | \ge 1/\tilde k
\end{equation}
hence
\begin{equation}
\inf_{x\in \mathbb S^2 : Z(x, t)(\omega) = u} |\nabla Z(x, t)(\omega) | \ge 1/\tilde k >0
\end{equation}
and $(\omega,t)\in A^T$.
\end{proof}

\begin{lemma}
\label{lemreg} Let $T>0$. There exists $\tilde \Omega_T \subseteq \Omega$, $%
\mathbb{P}(\tilde \Omega_T)=1$, such that for every $\omega\in \tilde
\Omega_T$ there exists $I_T(\omega)\subseteq [0,T]$, $\text{Leb}%
(I_T(\omega))=T$, such that the value $u$ is regular for $Z(\cdot,
t)(\omega) $ for every $t\in I_T(\omega)$.
\end{lemma}

In view of Lemma \ref{lemreg}, let us define $\tilde \Omega:=\cap_{n\in
\mathbb{N}} \tilde \Omega_n$ ($\mathbb{P}(\tilde \Omega)=1$), hence for
every $\omega\in \tilde \Omega$ there exists $I(\omega)\subseteq [0,+\infty)$
($\text{Leb}(I(\omega)^c) = 0$) such that the value $u$ is regular for $%
Z(\cdot, t)(\omega)$ for every $t\in I(\omega)$. On $\tilde \Omega$ we can
define the quantity
\begin{equation}  \label{rap_int}
\mathcal{C}_{T}(u)(\omega):=\,\int_{0}^{T}\,\Big(\mathcal{L}_{u}(t)(\omega)-%
\mathbb{E}[\mathcal{L}_{u}(t)]\Big)\,dt
\end{equation}
for every $T>0$, which is a random variable. 

\begin{proof}[Proof of Lemma \ref{lemreg}]
For $T>0$, we consider the measure space $([0,T], \mathcal B([0,T]), \text{Leb}_{[0,T]})$, and define
\begin{eqnarray*}
    A^T &:=& \lbrace (\omega, t)\in \Omega \times [0,T] : \text{ the value } u \text{ is regular for } Z(\cdot, t)(\omega)\rbrace.
\end{eqnarray*}
From Lemma \ref{lemmeas}, $A^T$ is measurable ($A^T \in \mathfrak F\otimes \mathcal B([0,T])$). In particular, the section $$A^T_t = \lbrace \omega\in \Omega  : \text{ the value } u \text{ is regular for } Z(\cdot, t)(\omega)\rbrace$$
is a measurable set and contains $\Omega_t$ ensuring that $\mathbb P(A^T_t)=1$.  A standard application of Fubini's theorem gives
\begin{equation}
  T =  \int_0^T \mathbb P(A^T_t)\,dt = \mathbb E \left [\int_0^T 1_{A^T}(\omega,t)\,dt
    \right ]
\end{equation}
implying that there exists $\tilde \Omega_T\subseteq \Omega$, $\mathbb P(\tilde \Omega_T) =1$, such that for every $\omega\in \tilde \Omega_T$ we have 
$
   \int_0^T 1_{A^T}(\omega,t)\,dt = T.
$
\end{proof}

\section{Square integrability}

\label{app-square}

In this Section first we prove that $\mathcal{C}_T(u)$ is square integrable.
By a standard application of Jensen's inequality and the stationarity of the
model we have
\begin{equation}
\mathbb{E}[\mathcal{C}_T(u)^2] \le T^2 \text{Var}(\mathcal{L}_u( 0))
\end{equation}
for any $T>0$. Hence it suffices to prove that $\mathcal{L}_u(0)$ is square
integrable (clearly, it is equivalent to show that $\mathcal{L}_u(t)$ is so,
for any $t\in\mathbb{R}$).

Recall the definition of $\epsilon$-approximating random variables $\mathcal{%
L}_u^\epsilon(t)$ in (\ref{Leps}).

\begin{lemma}
\label{approx2} As $\epsilon\to 0$,
\begin{equation}  \label{convergence}
\mathcal{L}_u^\epsilon(t) \to \text{length}(Z(\cdot, t)^{-1}(u))=\mathcal{L}%
_u(t)
\end{equation}
both a.s. and in $L^2(\mathbb{P})$.
\end{lemma}

\begin{proof}
The following conditions are satisfied:
\begin{enumerate}
\item \label{punto1} for every fixed $t\in \mathbb R$, the random field $Z(\cdot,t)$ is with probability one a
Morse function on $\mathbb S^2$, see Section \ref{subsec-object} for details;
\item \label{punto2} the covariance function $\Gamma$ of the field is at least twice continuously
differentiable with strictly positive second-order derivative in a
neighburhood of the origin, meaning that%
\begin{equation*}
\sum_{\ell }\ell ^{2}\frac{2\ell +1}{4\pi }C_{\ell }(0)<\infty \text{ .}
\end{equation*}%
Note that this expression is strictly positive unless $C_{\ell }(0)=0$ for all $%
\ell \geq 1.$
\end{enumerate}

 Also, \ref{punto2}. implies that the second derivative of the covariance
function is continuous at the origin, and hence%
\begin{equation*}
1-\Gamma (\cos \theta ,0)=\Gamma ^{\prime \prime }(0,0)\theta ^{2}+o(\theta
^{2})\text{ , as }\theta \rightarrow 0\text{ .}
\end{equation*}%
We recall incidentally that%
\begin{equation*}
\frac{\partial ^{2}}{\partial \theta ^{2}}\Gamma (\cos \theta )=\sum_{\ell }%
\frac{2\ell +1}{4\pi }C_{\ell }P_{\ell }^{\prime \prime }(\cos \theta )\sin
^{2}\theta -\sum_{\ell }\frac{2\ell +1}{4\pi }C_{\ell }P_{\ell }^{\prime
}(\cos \theta )\cos \theta
\end{equation*}%
which by Cauchy-Schwartz inequality has a unique maximum for $\theta =0,$ given by%
\begin{equation*}
\Gamma ^{\prime \prime }(0,0)=\sum_{\ell }\frac{2\ell +1}{4\pi }\frac{%
\lambda _{\ell }}{2}C_{\ell }\text{ .}
\end{equation*}
For notational simplicity we prove that the $L^{2}$-expansion holds at $u=0;$
the proof for different values is identical. Our argument is quite standard, see for instance \cite{MRW}.

We know that the boundary
length is defined almost-surely by
\begin{eqnarray*}
\mathcal{L}_{0}(t) &\mathcal{=}&\lim_{\varepsilon \rightarrow 0}\mathcal{L}%
_{0;\varepsilon }(t)\text{ , } \\
\mathcal{L}_{0;\varepsilon }(t) &:&=\int_{\mathbb{S}^{2}}\delta
_{\varepsilon }(Z(x,t))\left\Vert \nabla Z(x,t)\right\Vert dx\text{ ,}
\end{eqnarray*}%
where 
\begin{equation*}
\delta _{\varepsilon }(Z(x,t)):=\left\{
\begin{array}{c}
0\text{ for }x:Z(x,t)>\varepsilon  \\
\frac{1}{2\varepsilon }\text{ for }x:Z(x,t)\leq \varepsilon
\end{array}%
\right.
\end{equation*}%
and the almost-sure convergence follows from the standard arguments \cite[Lemma 3.1]{RW08}. Indeed, because $\delta _{\varepsilon }$ is integrable
and $Z(\cdot)$ is Morse we have, using the coarea formula for a fixed $t\in
\mathbb{R}$ (see i.e., \cite{adlertaylor}, p.169)%
\begin{equation*}
\int_{\mathbb{S}^{2}}\delta _{\varepsilon }(Z(x,t))\left\Vert \nabla
Z(x,t)\right\Vert dx=\int_{\mathbb{R}}\left\{ \int_{Z^{-1}(s,t)}\delta
_{\varepsilon }(Z(x,t))dx\right\} ds
\end{equation*}%
and thus we obtain%
\begin{equation*}
\int_{\mathbb{R}}\left\{ \int_{Z^{-1}(s,t)}\delta _{\varepsilon
}(Z(x,t))dx\right\} ds=\frac{1}{2\varepsilon }\int_{-\varepsilon
}^{\varepsilon }\textrm{length}\left[ Z^{-1}(s,t)\right] ds\rightarrow \textrm{length}\left[
Z^{-1}(0,t)\right] \text{ , as }\varepsilon \rightarrow 0\text{ ,}
\end{equation*}%
because the function $s\rightarrow \textrm{length}\left[ Z^{-1}(s,t)\right] $ is
continuous for Morse functions, see \ref{punto1}. In particular,
(\ref{convergence}) holds a.s.

We now want to show that the convergence occurs also in the $L^{2}$ sense;
because convergence holds almost surely, it is enough to show that
\begin{equation*}
\lim_{\varepsilon \rightarrow 0}\mathbb{E}\left[ \mathcal{L}_{0;\varepsilon
}^{2}(t)\right] =\mathbb{E}\left[ \mathcal{L}_{0}^{2}(t)\right] \text{ .}
\end{equation*}%
Indeed, we have that%
\begin{eqnarray*}
\lim_{\varepsilon \rightarrow 0}\mathbb{E}\left[ (\mathcal{L}_{0}(t)-%
\mathcal{L}_{0;\varepsilon }(t))^{2}\right]  &=&\lim_{\varepsilon
\rightarrow 0}\mathbb{E}\left[ (\mathcal{L}_{0}^{2}(t)+\mathcal{L}%
_{0;\varepsilon }^{2}(t)-2\mathcal{L}_{0}(t)\mathcal{L}_{0;\varepsilon }(t)%
\right]  \\
&=&2\mathbb{E}\left[ \mathcal{L}_{0}^{2}(t)\right] -2\lim_{\varepsilon
\rightarrow 0}\mathbb{E}\left[ \mathcal{L}_{0}\mathcal{(}t\mathcal{)L}%
_{0;\varepsilon }(t)\right] =0\text{ ,}
\end{eqnarray*}%
because by Fatou's Lemma and Cauchy-Schwartz inequality
\begin{equation*}
\mathbb{E}\left[ \mathcal{L}_{0}^{2}(t)\right] \leq \lim_{\varepsilon
\rightarrow 0}\inf \mathbb{E}\left[ \mathcal{L}_{0}(t)\mathcal{L}%
_{0;\varepsilon }(t)\right] \leq \lim_{\varepsilon \rightarrow 0}\sqrt{%
\mathbb{E}\left[ \mathcal{L}_{0}^{2}(t)\right] \mathbb{E}\left[ \mathcal{L}%
_{0;\varepsilon }^{2}(t)\right] }=\mathbb{E}\left[ \mathcal{L}_{0}^{2}(t)%
\right] \text{ .}
\end{equation*}%
Note that, by the coarea formula%
\begin{eqnarray*}
\mathbb{E}\left[ \mathcal{L}_{0;\varepsilon }^{2}(t)\right]  &=&\mathbb{E}%
\left[ \left\{ \int_{\mathbb{S}^{2}}\left\{ \delta _{\varepsilon
}(Z(x,t))\left\Vert \nabla Z(x,t)\right\Vert \right\} dx\right\} ^{2}\right]
\\
&=&\mathbb{E}\left[ \left\{ \int_{\mathbb{R}}\int_{Z(x,t)=u}\delta
_{\varepsilon }(Z(x,t))dxdu\right\} ^{2}\right]  \\
&=&\mathbb{E}\left[ \left\{ \int_{\mathbb{R}}\mathcal{L}_{u}(t)\delta
_{\varepsilon }(u)du\right\} ^{2}\right] \text{ ,}
\end{eqnarray*}%
where as before by $\mathcal{L}_{u}(t)$ we denote the length of the set $%
Z(x,t)=u.$ We can now show that the application $u\rightarrow \mathbb{E}%
\left[ \mathcal{L}_{u}^{2}(t)\right] $, or more explicitly%
\begin{equation*}
\mathbb{E}\left[ \mathcal{L}_{u}^{2}(t)\right] =\int_{\mathbb{S}^{2}\times
\mathbb{S}^{2}}\mathbb{E}\left[ \left. \left\Vert \nabla
Z(x_{1},t)\right\Vert \left\Vert \nabla Z(x_{2},t)\right\Vert \right\vert
Z(x_{1},t)=u,Z(x_{2},t)=u\right] \phi
_{Z(x_{1},t),Z(x_{2},t)}(u,u)dx_{1}dx_{2}
\end{equation*}%
\begin{equation*}
=8\pi ^{2}\int_{0}^{\pi }\text{ }\mathbb{E}\left[ \left. \left\Vert \nabla
Z(N,t)\right\Vert \left\Vert \nabla Z(y(\theta ),t)\right\Vert \right\vert
Z(N,t)=u,Z(y(\theta ),t)=u\right] \phi _{Z(N,t),Z(y(\theta ),t)}(u,u)\sin
\theta d\theta \text{,}
\end{equation*}%
is continuous. The integrand function is obviously continuous in $u,$ and
thus to check the latter statement it is enough to use Dominated Convergence
Theorem. We first note that%
\begin{eqnarray*}
\phi _{Z(N,t),Z(y(\theta ),t)}(u,u)\sin \theta  &\leq &\phi
_{Z(N,t),Z(y(\theta ),t)}(0,0)\sin \theta  \\
&=&\frac{1}{2\pi \sqrt{1-\Gamma ^{2}(\cos \theta ,0)}}\sin \theta =O(1)\text{
,}
\end{eqnarray*}%
uniformly over $\theta $, because%
\begin{equation*}
1-\Gamma ^{2}(\cos \theta ,0)=(1+\Gamma (\cos \theta ,0))(1-\Gamma (\cos
\theta ,0))\geq \frac{c}{\theta ^{2}}\text{ .}
\end{equation*}%
On the other hand, to evaluate%
\begin{equation*}
\mathbb{E}\left[ \left. \left\Vert \nabla Z(x_{1},t)\right\Vert \left\Vert
\nabla Z(x_{2},t)\right\Vert \right\vert Z(Nt)=u,Z(y(\theta ),t)=u\right]
\end{equation*}%
we can use Cauchy-Schwartz inequality, and bound
\begin{eqnarray*}
&&\mathbb{E}\left[ \left. w_{i}^{2}\right\vert Z(N,t)=u,Z(y(\theta ),t)=u%
\right]  \\
&=&Var\left[ \left. w_{i}\right\vert Z(N,t)=u,Z(y(\theta ),t)=u\right]
+\left\{ \mathbb{E}\left[ \left. w_{i}\right\vert Z(N,t)=u,Z(y(\theta ),t)=u%
\right] \right\} ^{2},
\end{eqnarray*}%
for $i=1,2,3,4,$ where%
\begin{equation*}
\left(
\begin{array}{c}
w_{1} \\
w_{2} \\
w_{3} \\
w_{4}%
\end{array}%
\right) :=\left(
\begin{array}{c}
\nabla Z(x_{1},t) \\
\nabla Z(x_{2}t)%
\end{array}%
\right) \text{ .}
\end{equation*}%
It is a standard fact for Gaussian conditional distributions that%
\begin{equation*}
Var\left[ \left. w_{i}\right\vert Z(N,t)=u,Z(y(\theta ),t)=u\right] \leq Var%
\left[ w_{i}\right] <\infty \text{ .}
\end{equation*}%
Similarly, standard results on Gaussian conditional expectations give
(compare \cite{Wig}, Appendix A)%
\begin{equation*}
\mathbb{E}\left[ \left.
\begin{array}{c}
w_{1} \\
w_{2} \\
w_{3} \\
w_{4}%
\end{array}%
\right\vert Z(N,t)=u,Z(y(\theta )t)=u\right] =B_{\ell }^{T}(\theta )A_{\ell
}^{-1}(\theta )\left(
\begin{array}{c}
u \\
u%
\end{array}%
\right) \text{ ,}
\end{equation*}%
where%
\begin{eqnarray*}
B_{\ell }^{T}(\theta ) &=&\left(
\begin{array}{cc}
\sum_{\ell }\frac{2\ell +1}{4\pi }C_{\ell }P_{\ell }^{\prime }(\cos \theta
)\sin \theta  & 0 \\
0 & 0 \\
0 & \sum_{\ell }\frac{2\ell +1}{4\pi }C_{\ell }P_{\ell }^{\prime }(\cos
\theta )\sin \theta  \\
0 & 0%
\end{array}%
\right) \text{ , } \\
A_{\ell }^{-1}(\theta ) &=&\frac{1}{1-\Gamma ^{2}(\cos \theta ,0)}\left(
\begin{array}{cc}
1 & -\Gamma (\cos \theta ,0) \\
-\Gamma (\cos \theta ,0) & 1%
\end{array}%
\right) \text{ ,}
\end{eqnarray*}%
That we obtain for the the conditional expected value%
\begin{equation*}
\frac{1}{1-\Gamma ^{2}(\cos \theta ,0)}\left(
\begin{array}{cc}
-\sum_{\ell }\frac{2\ell +1}{4\pi }C_{\ell }P_{\ell }^{\prime }(\cos \theta
)\sin \theta  & \Gamma (\cos \theta )\sum_{\ell }\frac{2\ell +1}{4\pi }%
C_{\ell }P_{\ell }^{\prime }(\cos \theta )\sin \theta  \\
0 & 0 \\
-\Gamma (\cos \theta )\sum_{\ell }\frac{2\ell +1}{4\pi }C_{\ell }P_{\ell
}^{\prime }(\cos \theta )\sin \theta  & -\sum_{\ell }\frac{2\ell +1}{4\pi }%
C_{\ell }P_{\ell }^{\prime }(\cos \theta )\sin \theta  \\
0 & 0%
\end{array}%
\right) \left(
\begin{array}{c}
u \\
u%
\end{array}%
\right)
\end{equation*}%
\begin{equation*}
=\frac{1}{1-\Gamma ^{2}(\cos \theta )}\left(
\begin{array}{c}
u(\Gamma (\cos \theta ,0)-1)\sum_{\ell }\frac{2\ell +1}{4\pi }C_{\ell
}P_{\ell }^{\prime }(\cos \theta )\sin \theta  \\
0 \\
u(1-\Gamma (\cos \theta ,0))\sum_{\ell }\frac{2\ell +1}{4\pi }C_{\ell
}P_{\ell }^{\prime }(\cos \theta )\sin \theta  \\
0%
\end{array}%
\right)
\end{equation*}%
\begin{equation*}
=\frac{1}{1+\Gamma (\cos \theta ,0)}\left(
\begin{array}{c}
-u\sum_{\ell }\frac{2\ell +1}{4\pi }C_{\ell }P_{\ell }^{\prime }(\cos \theta
)\sin \theta  \\
0 \\
u\sum_{\ell }\frac{2\ell +1}{4\pi }C_{\ell }P_{\ell }^{\prime }(\cos \theta
)\sin \theta  \\
0%
\end{array}%
\right) \text{ .}
\end{equation*}%
This vector function is immediately seen to be uniformly bounded over $%
\theta ,$ whence the Dominated Convergence Theorem holds.
To conclude the proof, we note that%
\begin{eqnarray*}
\mathbb{E}\left[ \mathcal{L}_{0}^{2}(t)\right]  &\leq &\lim
\inf_{\varepsilon \rightarrow 0}\mathbb{E}\left[ \left\{ \int_{\mathbb{S}%
^{2}}\left\{ \delta _{\varepsilon }(Z(x,t))\left\Vert \nabla
Z(x,t)\right\Vert \right\} dx\right\} ^{2}\right]  \\
&=&\lim \inf_{\varepsilon \rightarrow 0}\mathbb{E}\left[ \mathcal{L}%
_{0;\varepsilon }^{2}(t)\right] \leq \lim \sup_{\varepsilon \rightarrow 0}%
\mathbb{E}\left[ \mathcal{L}_{0;\varepsilon }^{2}(t)\right]
\end{eqnarray*}%
(by Fatou's Lemma and definitions) and then%
\begin{eqnarray*}
&=&\lim \sup_{\varepsilon \rightarrow 0}\mathbb{E}\left[ \left\{ \int_{%
\mathbb{S}^{2}}\left\{ \delta _{\varepsilon }(Z(x,t))\left\Vert \nabla
Z(x,t)\right\Vert \right\} dx\right\} ^{2}\right]  \\
&=&\lim \sup_{\varepsilon \rightarrow 0}\mathbb{E}\left[ \left\{ \int_{%
\mathbb{R}}\mathcal{L}_{u}(t)\delta _{\varepsilon }(u)du\right\} ^{2}\right]
\end{eqnarray*}%
(by co-area formula) and%
\begin{equation*}
\leq \lim \sup_{\varepsilon \rightarrow 0}\int_{\mathbb{R}}\mathbb{E}\left[
\mathcal{L}_{u}^{2}(t)\right] \delta _{\varepsilon }(u)du=\mathbb{E}\left[
\mathcal{L}_{0}^{2}(t)\right] \text{ ,}
\end{equation*}%
by Cauchy-Schwartz, the definition of the $\delta $ function and continuity
of the application $u\rightarrow \mathbb{E}\left[ \mathcal{L}_{u}^{2}(t)%
\right] $. We have thus shown that $\mathbb{E}\left[ \mathcal{L}%
_{\varepsilon }^{2}(t)\right] \rightarrow \mathbb{E}\left[ \mathcal{L}%
_{0}^{2}(t)\right] ,$ and the proof is completed.
\end{proof}

\section{Chaotic decomposition}\label{app-chaos}

We need the following standard technical result, adapted from the nodal case \cite{MPRW16} to any threshold $u\in \mathbb R$.

\begin{lemma}\label{indicator function}
The following decomposition holds in $L^2(\Omega)$
\begin{equation*}
\frac{1}{2\varepsilon}{1}_{[u-\varepsilon ,u+\varepsilon ]}(Z
)=\sum_{l=0}^{+\infty }\frac{1}{l!}\beta_l^{\varepsilon}(u)\,H_{l}(Z),
\end{equation*}
where $Z\sim \mathcal N(0,1)$, and for $l \geq 1$
\[
\beta_l^{\varepsilon}(u)=-\frac{1}{2\varepsilon}\left (
\phi \left (u+\varepsilon \right) H_{l-1} \left (u+\varepsilon \right)-
\phi \left (u-\varepsilon \right) H_{l-1} \left (u-\varepsilon \right) \right ),
\]
while for $l=0$
\[
\beta_0^{\varepsilon} = \frac{1}{2\varepsilon} \int_{u-\varepsilon}^{u+\varepsilon} \phi(t)\,dt.
\]
Moreover, as $\varepsilon\to 0$,
$$
\beta_{l}^\varepsilon(u) \to \beta_l(u),
$$
where $\beta_l(u)$ coincides with (\ref{e:beta}) for every $l$.
\end{lemma}
We are now ready to establish the chaotic decomposition of the average boundary length.

\begin{proof}[Proof of Proposition \ref{teoexpS}]
For fixed $x\in \mathbb S^2$, $t\in \mathbb R$,
the projection of the random variable
$$
\frac{1}{2\varepsilon} 1_{[u-\epsilon, u+\epsilon]}(Z(x,t)) \|\widetilde \nabla Z(x,t)\|
$$
onto the chaos $C_{q}$, for $q\geq 0$, equals
$$
\sum_{m=0}^{q}\sum_{k=0}^{m} \frac{\alpha _{k,m-k}\beta^\epsilon _{q-m}(u)}{%
(k)!(m-k)!(q-m)!}
H_{q-m}(Z(x,t))H_{k}(\widetilde\partial_{1;x} Z(x,t))H_{m-k}(\widetilde{\partial}%
_{2;x} Z(x,t)),
$$
where $\lbrace \beta_l^\epsilon(u)\rbrace_l$ is the collection of chaotic coefficients found in
Lemma \ref{indicator function}.
Since $\int_{0,T]}\int_{\mathbb{S}^2}\,dxdt<\infty$, standard arguments based on Jensen's
inequality and dominated convergence yield that $\mathcal C_T^\epsilon(u)[0] = 0$ while for $q\ge 1$
\begin{eqnarray*}
\mathcal C_T^{\varepsilon}(u)[q]= \sum_{m=0}^{q}\sum_{k=0}^{m} \frac{\alpha _{k,m-k}\beta^\epsilon _{q-m}(u)}{%
(k)!(m-k)!(q-m)!} \int_{0,T]}\int_{\mathbb{S}^2} H_{q-m}(Z(x,t))H_{k}(\widetilde\partial_{1;x} Z(x,t))H_{m-k}(\widetilde{\partial}%
_{2;x} Z(x,t))\,dx dt
\end{eqnarray*}
in $L^2(\Omega)$.
In view of Lemma \ref{approx2} and Lemma \ref{indicator function}, the random variable $\mathcal C_T(u)$ being in the Wiener chaos, one has that for every $q$,
as $\varepsilon \to 0$, $\mathcal C_T^\varepsilon(u)[q]$
necessarily converge to the $q$-th chaotic component of $\mathcal C_T(u)$, that is, $\mathcal C_T(u)[q]$ as in (\ref{e:ppS}), still in $L^2(\Omega)$.
\end{proof}

\

\noindent Dipartimento di Matematica, Universit\`{a} degli Studi di Roma
``Tor Vergata''\newline
E-mail address: \texttt{marinucc@mat.uniroma2.it}

\medskip

\noindent Dipartimento di Matematica e Applicazioni, Universit\`{a} degli
Studi di Milano-Bicocca\newline
E-mail address: \texttt{maurizia.rossi@unimib.it}

\medskip

\noindent Dipartimento di Matematica e Applicazioni, Universit\`{a} degli Studi di
Napoli ``Federico II''\newline
E-mail address: \texttt{anna.vidotto@unina.it}

\end{document}